\documentclass[a4paper, 11pt]{article}

\usepackage[utf8]{inputenc}
\usepackage[T1]{fontenc}
\usepackage[a4paper,left=3cm,right=3cm,top=3cm,bottom=3cm]{geometry}

\usepackage{amssymb}
\usepackage{amsmath}
\usepackage{mathtools}
\usepackage{amsthm}
\usepackage{bbm}
\usepackage{hyperref}
\usepackage{xurl}
\usepackage{authblk}

\usepackage[round, sort]{natbib}
\usepackage[bibliography=common]{apxproof}
\renewcommand{\appendixsectionformat}[2]{Appendix}

\usepackage[deletedmarkup=it]{changes}
\makeatletter
\AddToHook{cmd/added/before}{\def\Changes@AuthorColor{blue}}
\AddToHook{cmd/deleted/before}{\def\Changes@AuthorColor{red}}
\AddToHook{cmd/replaced/before}{\def\Changes@AuthorColor{green}}
\makeatother

\theoremstyle{plain}
\newtheorem{thm}{Theorem}[section]
\newtheorem{cor}[thm]{Corollary}
\newtheorem{lem}[thm]{Lemma}
\newtheorem{prop}[thm]{Proposition}

\theoremstyle{definition}
\newtheorem{defn}[thm]{Definition}

\theoremstyle{remark}

\newcommand{\R}{\mathbb{R}}
\newcommand{\Pb}{\mathbb{P}}
\newcommand{\E}{\mathbb{E}}

\newcommand{\Oc}{\mathcal{O}}

\newcommand{\Var}{\mathrm{Var}}

\newcommand{\id}{\mathrm{Id}}

\newcommand{\mathd}{\mathrm{d}}
\newcommand{\tv}{\textsc{TV}}

\newcommand{\scp}[2]{\left\langle #1, #2 \right\rangle}
\newcommand\abs[1]{\left \lvert#1\right \rvert}
\newcommand\norm[1]{\left\Vert#1\right\Vert}

\newcommand{\ls}[1]{\ensuremath{\mathrm{LSI}_{#1}}}
\newcommand{\po}[1]{\ensuremath{\mathrm{PI}_{#1}}}

\newcommand{\cls}[1]{\ensuremath{\mathtt{s}_{#1}}}
\newcommand{\cpo}[1]{\ensuremath{\mathtt{p}_{#1}}}

\begin{document}
	
	\title{Mixing Time Bounds for the Gibbs Sampler\\ under Isoperimetry}

	\author[1]{Alexander Goyal}
	\author[2]{George Deligiannidis}
	\author[1]{Nikolas Kantas}
	
	\affil[1]{Department of Mathematics, Imperial College London}
	\affil[2]{Qube Research \& Technologies}
	
	\maketitle

	\begingroup
	\leftskip3em
	\rightskip\leftskip
	\hyphenpenalty=10000
	\small
	
	\begin{center}
		\textbf{Abstract}
	\end{center}

	We establish bounds on the conductance for the systematic-scan and random-scan Gibbs samplers when the target distribution satisfies a Poincar\'e or log-Sobolev inequality and possesses sufficiently regular conditional distributions.
	These bounds lead to mixing time guarantees that extend beyond the log-concave setting, offering new insights into the convergence behavior of Gibbs sampling in broader regimes.
	Moreover, we demonstrate that our results remain valid for log-Lipschitz and log-smooth target distributions. Our approach relies on novel  isoperimetric inequalities and a sequential coupling argument for the Gibbs sampler.

	\endgroup
	
	\section{Introduction}

	Sampling from a target distribution has long been a central problem in many fields related to statistics and applied probability. Markov chain Monte Carlo (MCMC) methods are among the most popular algorithms for generating samples from a potentially high-dimensional target distribution $\pi$. 
	The idea is to simulate an ergodic Markov chain that admits $\pi$ as its invariant distribution, and use the states of the chain as approximate samples from $\pi$.
	The Gibbs sampler
	\citep{geman1984stochastic} is one of the most widely used MCMC methods in practical applications.
	This algorithm updates a single coordinate (or a block of coordinates) at each iteration using its conditional distribution under $\pi$, while keeping the other coordinates fixed.
	There are two common variants: systematic-scan Gibbs and random-scan Gibbs.
	In the systematic-scan case, coordinates are updated one by one in a fixed order.
	In random-scan Gibbs, one coordinate is chosen randomly at each step to be updated.
	The advantage of Gibbs sampling is that it simplifies high-dimensional sampling by reducing it to a series of lower-dimensional updates. Moreover, in many popular settings, such as hierarchical or exponential family models, the conditional distributions are tractable, eliminating the need for step-size selection.
	
	Given the popularity of Gibbs sampling in practical applications, understanding the convergence properties of this algorithm has been of particular importance.
	Under mild conditions on $\pi$, qualitative convergence for both variants of the Gibbs sampler was established as early as in \citet{roberts1994simple}.
	However, early literature on the quantitative convergence properties of the Gibbs sampler relied on strong assumptions regarding $\pi$.
	Exact convergence rates were established for multivariate Gaussian distributions \citep{amit1991comparing, amit1996convergence, roberts1997updating}, as well as for models involving Dobrushin coefficients \citep{wu2006poincare, dyer2008dobrushin, wang2014convergence, wang2017convergence, wang2019convergence}.
	More recently, \citet{yang2023complexity} generalized standard convergence tools based on minorization and Lyapunov drift conditions for Gibbs samplers, yielding new results for Bayesian hierarchical models \citep{qin2019convergence, qin2022wasserstein}; see also \citet{biswas2022coupling} for similar results using coupling techniques.
	For such models, \citet{ascolani2024dimension} established 
	further bounds on the convergence rate of the Gibbs sampler, leveraging Bayesian asymptotics under specific random data-generating assumptions.
	
	In this paper, we characterize the convergence of Gibbs sampling using the so-called mixing time, which quantifies the number of iterations needed for the Markov chain to be within a specified distance of its invariant distribution.
	A common technique for bounding the mixing time is to analyze a quantity known as the conductance of the Markov chain. 
	One of the key components of this technique relies on the notion of isoperimetry for the target distribution.
	This is often formalized via Cheeger's  isoperimetric inequality \citep{cheeger1970lower}, which holds for log-concave target distributions \citep{bobkov1999isoperimetric} and remains valid under bounded perturbations, as characterized by the criterion of \citet{holley1987logarithmic}.
	The conductance approach has proven successful for the Hit-and-Run algorithm \citep{lovasz1993random, lovasz1999hit, lovasz2004hit, vempala2005geometric}, as well as for random-walk Metropolis-type algorithms \citep{dwivedi2019log, mangoubi2019nonconvex, chen2020fast, chewi2021optimal, zou2021faster, wu2022minimax, mou2022efficient, chen2023simple}, among others.
	However, conductance-based proof techniques have seen limited use in the context of Gibbs sampling, with notable exceptions including sampling from uniform distributions over convex bodies \citep{narayanan2022mixing, laddha2023convergence, narayanan2024sampling}, posterior distributions in Bayesian regression models \citep{lee2024fast}, and strongly log-concave targets \citep{wadia2024mixing}.
	
	In the recent seminal work by \citet{andrieu2024explicit}, the authors unify many ideas from the references mentioned above and provide clear guidelines for controlling the mixing time of general MCMC methods under the assumption of isoperimetry for the target distribution.
	Specifically, they consider log-concave target distributions which also satisfy either a Poincaré or a log-Sobolev inequality.
	However, log-concavity is not necessary for these functional inequalities to hold: a wide range of non-log-concave distributions still satisfy a Poincaré or log-Sobolev inequality, making the class of distributions satisfying such functional inequalities substantially broader.
	Nevertheless, when deriving explicit bounds on the mixing time of the random-walk Metropolis algorithm, even these authors exclusively focus on the strongly log-concave case.
	
	For several other classes of sampling algorithms, convergence bounds under a Poincar\'e or log-Sobolev inequality have been established without relying on Cheeger's isoperimetric inequality. 
	These include discretizations of the overdamped Langevin diffusion \citep{altschuler2024shifted, chewi2025analysis, mou2022improved, balasubramanian2022towards, kandasamy2024poisson, mitra2025fast, vempala2019rapid, anari2024fast}, discretizations of the underdamped Langevin diffusion \citep{altschuler2025shifted, lehec2025convergence, ma2021there, zhang2023improved, camrud2023second, kandasamy2024poisson, anari2024fast}, as well as proximal samplers \citep{wibisono2019proximal, chen2022improved, fan2023improved, liang2022proximal, mitra2025fast}, among others. 
	In these works, convergence is typically established by comparing the discrete-time sampling algorithm to an associated continuous-time process, such as the Langevin diffusion, whose convergence properties are well understood in terms of functional inequalities. 
	In contrast, the analysis of \citet{andrieu2024explicit} is carried out directly at the level of the discrete-time Markov chain, without relying on an explicit comparison with a continuous-time diffusion.
	For the Gibbs sampler, this distinction is especially relevant, as there does not appear to exist a natural continuous-time process against which the algorithm can be compared.
	It is therefore interesting to investigate whether the framework developed by \citet{andrieu2024explicit} can be applied to the Gibbs sampler targeting non-log-concave distributions, specifically those satisfying either a Poincar\'e or log-Sobolev inequality.
	
	Another central notion in \citet{andrieu2024explicit} is that of a so-called close coupling condition for a Markov chain, which provides a quantitative link between geometry and probability.
	Informally, it requires that whenever two points are ``close'' in an appropriate metric, the distributions obtained after one step of the Markov chain from these points are ``close'' in total variation distance.
	The importance of such connections between geometric and probabilistic distances in the context of mixing times goes back at least to the work of \citet[Section 5]{lovasz1999hit}, which studies them in the context of the Hit-and-Run algorithm on convex bodies.
	In \citet{andrieu2024explicit}, this idea is formalized into a condition that, together with isoperimetry, yields explicit mixing time bounds directly at the level of the discrete-time Markov chain.
	
	\subsection{Our contributions} 
	
	In this paper, we establish conductance bounds for the systematic-scan and random-scan Gibbs samplers over a broad class of target distributions supported on $\R^d$.
	In particular, consider a target distribution $\pi$ satisfying an $L^q$-Poincar\'e inequality with constant $\cpo{q}$ and whose one-dimensional conditional distributions are H\"older continuous in total variation distance with constant $M$ and exponent $\beta$ (see Subsection \ref{subsec:assumptions} for formal definitions). 
	Denote by $P_\mathrm{SS}$ and $P_\mathrm{RS}$ the Markov kernels corresponding to the systematic-scan and random-scan Gibbs samplers, respectively.
	Then, for sufficiently large $d$, our main results (Theorem \ref{thm:condss} and Corollary \ref{cor:condrs}) provide the following bounds on the conductance of the lazy systematic-scan Gibbs kernel, $\Phi(\widetilde{P}_\mathrm{SS})$, and the random-scan Gibbs kernel, $\Phi(P_\mathrm{RS})$:
	\begin{gather*}
		\Phi (\widetilde{P}_\mathrm{SS}) \gtrsim \frac{\cpo{q}}{2^{2q}(Md)^{q/\beta}}, \\
		\Phi (P_{\mathrm{RS}} ) \gtrsim \frac{\cpo{q}}{2^{2q} (3M)^{q/\beta}(d\cdot \log d)^{q/\beta+1}},
	\end{gather*}
	where $\gtrsim$ denotes inequalities that hold up to a universal constant. Bounds on the mixing time in total variation distance then follow from results of \citet{lovasz1993random}.

	To prove the conductance bounds, we establish novel intermediate results that may also be of independent interest.
	\begin{itemize}
		\item  We prove new three-set isoperimetric inequalities for distributions satisfying an $L^q$-Poincar\'e or $L^q$-log-Sobolev inequality, without assuming log-concavity (Propositions \ref{3siPI} and \ref{3siLSI}). 
		This is achieved by applying these functional inequalities to a carefully chosen test function.
		\item Assuming that the one-dimensional conditional distributions under $\pi$ are H\"older continuous in total variation distance, we derive close coupling conditions for the kernels $P_\mathrm{SS}$ and $P^{N_d}_\mathrm{RS}$, where $N_d \sim d \cdot \log d$ (Propositions \ref{ssgibbsccc} and \ref{rsgibbsccc}).
		These conditions are established through sequential coupling arguments.
		Here, $P^{N_d}_\mathrm{RS}$ denotes the random-scan Gibbs kernel iterated $N_d$ times, yielding a kernel that is supported on the entire space $\R^d$.
		\item We show that if $\pi$ is log-Lipschitz continuous or log-smooth, then the one-dimensional conditional distributions under $\pi$ are H\"older continuous in total variation distance (Lemmas \ref{tvconditionalsloglip} and \ref{tvconditionalslogsmooth}). In the log-smooth case, this is shown under the additional assumption that the one-dimensional conditional distributions of $\pi$ satisfy a uniform $L^2$-Poincar\'e inequality.
		The proof relies on the fact that log-Lipschitz continuity and log-smoothness (under the above assumption) are inherited by the marginal distributions of $\pi$, as shown in Lemma \ref{marginalsloglipsmooth}.
	\end{itemize}
	After establishing the three-set isoperimetric inequalities and close coupling conditions, bounds on the conductance follow from Proposition \ref{conductanceTHM} below, which is a slight modification of \citet[Lemma 15]{andrieu2024explicit}.
	
	\subsection{Connections with recent work}

	As a special case, the conductance approach has been successfully applied to the random-scan Gibbs sampler targeting uniform distributions over convex bodies, where quantitative convergence guarantees have been established \citep{narayanan2022mixing, laddha2023convergence, narayanan2024sampling}.
	\citet{narayanan2022mixing} establish a close coupling condition for a multi-step Metropolis-within-Gibbs scheme, where the corresponding Markov kernel updates a randomly chosen coordinate using a single step of random-walk Metropolis while keeping the others fixed.	
	They then compare the conductance of this auxiliary Markov chain with that of the original random-scan Gibbs sampler.
	\citet{laddha2023convergence} take a different approach, proving a three-set isoperimetric inequality for partitions of space into axis-disjoint regions. This allows them to control the mixing behavior without relying on auxiliary chains.
	
	More recently, \citet{narayanan2024sampling} refine these results by reducing the dependence on the initialization (the ``warm start'' parameter) of the mixing time bounds from polynomial to logarithmic.
	Their approach decomposes the convex body into axis-aligned dyadic cubes and establishes an isoperimetric inequality with respect to a metric that magnifies distances near the boundary.
	The dimension dependence of their mixing time bounds can be further improved using the sharper isoperimetric inequalities of \citet{fernandez2023ell_0}.
	In this work, we develop a complementary approach for target distributions supported on $\R^d$ and establish three-set isoperimetric inequalities with respect to the Euclidean metric that hold for arbitrary partitions. However, it is unclear whether the results in \citet{narayanan2022mixing,laddha2023convergence,narayanan2024sampling} for sampling from convex bodies can be recovered from our framework.
	
	Additionally, the conductance approach has been used to establish rapid mixing of the Gibbs sampler targeting the posterior distributions in Bayesian probit, logit, and lasso regression models \citep{lee2024fast}. 
	In each case, the Gibbs sampler is made tractable by augmenting the state space with latent variables.
	For such data-augmented chains, \citet{lee2024fast} show that close coupling conditions follow directly from the conditional independence structure of the posterior distribution.
	Isoperimetric inequalities are obtained via transference inequalities to a transformed Markov chain with a log-concave target distribution.
	Under suitable assumptions on the data-generating process, these authors then establish bounds on the mixing time that hold with high probability.

	More generally, mixing time bounds for the random-scan Gibbs sampler targeting log-smooth and strongly log-concave distributions have recently been established by \citet{ascolani2024entropy, wadia2024mixing}.
	In particular, \citet{ascolani2024entropy} prove that the random-scan Gibbs sampler contracts in relative entropy and provide a sharp characterization of the contraction rate.
	Their approach relies on triangular transport maps to decompose relative entropy into coordinate-wise terms. 
	\citet{wadia2024mixing} follow a similar approach to \citet{laddha2023convergence}, deriving a three-set isoperimetric inequality for partitions of space into axis-disjoint regions in the setting of general log-smooth and strongly log-concave distributions.
	The key insight behind their approach is that such distributions are approximately uniform on small enough domains.
	Both proof strategies rely heavily on the log-smoothness and log-concavity of $\pi$, so it is unclear whether they can be extended to more general settings, such as when the target distribution satisfies a Poincaré or log-Sobolev inequality.
	
	\subsection{Organization} 
	
	The rest of this paper is structured as follows.
	Section \ref{sec:prelim} introduces the Gibbs sampler and outlines the relevant assumptions on the target distribution.
	Section \ref{sec:markovconv} reviews standard definitions and key results on Markov chain mixing times. In particular, we explain how a three-set isoperimetric inequality and a close coupling condition can be used to derive lower bounds on conductance.
	Next, we establish three-set isoperimetric inequalities for target distributions satisfying a Poincar\'e or log-Sobolev inequality (Subsection \ref{sec:3sifi}) and derive close coupling conditions for the systematic-scan and random-scan Gibbs kernels (Subsection \ref{sec:ccss}). We then use these results to prove conductance and mixing time bounds for both variants of the Gibbs sampler (Subsection \ref{sec:mixbounds}).
	Finally, Section \ref{sec:extensionapp} extends these bounds to log-Lipschitz continuous and log-smooth target distributions.
	We also provide examples of non-log-concave target distributions for which our results can be applied to establish mixing time bounds for both variants of the Gibbs sampler, filling a gap in the existing literature.

	\section{Preliminaries} \label{sec:prelim}
	
	\subsection{The Gibbs sampler} \label{subsec:gs}

	We consider a probability distribution $\pi (\mathd x)$ supported on $\R^d$, equipped with its standard Borel $\sigma$-algebra. Let $P(x,\mathd y)$ be the Markov transition kernel of a discrete-time Markov chain on $\R^d$ with invariant distribution $\pi$. Here, the Markov chain is interpreted as the MCMC method for generating approximate samples from $\pi$. The kernel $P$ is said to be reversible with respect to $\pi$ if
	$$
	\pi ( \mathd x) P(x,\mathd y) = \pi(\mathd y) P(y, \mathd x) \qquad \forall x,y \in \R^d.
	$$
	Reversibility with respect to $\pi$ implies that $\pi$ is invariant, however the reverse implication is in general false. 
	Among MCMC methods, we focus on the Gibbs sampler. This method has two commonly used variants: the random-scan Gibbs sampler and the systematic-scan Gibbs sampler, which we now introduce.
	Let  $x = \left(x_1, \dots, x_d\right) \in \R^d$. For $1 \leq k \leq l \leq d$, we write
	$$
	x_{k:l} = (x_{k}, x_{k+1}, \dots, x_{l-1}, x_{l}) \in \R^{l-k+1}
	$$
	and
	$$x_{-k} = (x_1, \dots, x_{k-1},  x_{k+1}, \dots, x_d) \in \R^{d-1}.$$ 
	Given another point $y = \left(y_1, \dots, y_d\right) \in \R^d$, we also write
	\begin{equation}\label{notationindex}
		(x_{-k}, y_k) = (x_1, \dots, x_{k-1}, y_k, x_{k+1}, \dots, x_d).
	\end{equation}
	We additionally use the notation $\left[ d \right] = \left\{1,\dots, d\right\}$. Let $X \sim \pi$, that is, $X$ is distributed according to $\pi$. Denote the conditional distribution of $X_k$ given $X_{-k} = x_{-k}$ by $\pi( \cdot \mid x_{-k})$. For all $k \in [d]$, define the Markov kernel
	\begin{equation} \label{kernelPk}
		P_k (x, \mathd y) = \pi ( \mathd y_k \mid x_{-k}) \delta_{x_{-k}} (\mathd y_{-k}),
	\end{equation}
	where $\delta_{x_{-k}}$ denotes the $(d-1)$-dimensional Dirac delta distribution: for all measurable $S \subseteq \R^{d-1}$, 
	$$
	\delta_{x_{-k}} (S) = 
	\begin{cases}
		1  &\mathrm{if }~ x_{-k} \in S, \\
		0  &\mathrm{if }~ x_{-k} \notin S.
	\end{cases}
	$$
	Using the kernels $(P_k)_{k\in[d]}$, we now define the Markov kernels corresponding to the random-scan and systematic-scan Gibbs samplers. For two Markov kernels $P$ and $Q$, we denote their composition by $PQ(x,\mathd y) = \int Q(u, \mathd y) P(x, \mathd u)$. With this convention, $PQ$ corresponds to first updating the chain according to $P$ and then according to $Q$.
	
	\begin{defn}
		The Markov kernel associated with the random-scan Gibbs sampler is defined as the mixture of the Markov kernels $(P_k)_{k\in[d]}$:
		$$
		P_{\mathrm{RS}} = \frac{1}{d} \sum_{k=1}^d P_k.
		$$
		The Markov kernel associated with the systematic-scan Gibbs sampler is defined as the composition of the Markov kernels $(P_k)_{k\in[d]}$:
		$$
		P_{\mathrm{SS}} = P_1 P_{2} \dots P_{d-1} P_d.
		$$
	\end{defn}
	
	In words, the random-scan Gibbs sampler selects a single coordinate uniformly at random and resamples it from its conditional distribution, while leaving the other coordinates unchanged. 
	In contrast, the  systematic-scan Gibbs sampler sequentially updates each coordinate in a predetermined order, from 1 to $d$, resampling each from its respective conditional distribution. 
	
	Naturally, both variants of the Gibbs sampler are $\pi$-invariant. However, the Markov kernel associated with the random-scan Gibbs sampler is reversible with respect to $\pi$, whereas the Markov kernel associated with the systematic-scan Gibbs sampler is generally not.  
	To address this, one approach is to use the multiplicative reversibilization of the systematic-scan Gibbs sampler. The kernel for this version is defined as
	$$
		P_\mathrm{SS} P_{\mathrm{SS}}^* = P_1 P_{2} \dots P_{d-1}  P_d^2 P_{d-1} \dots P_2 P_1,
	$$
	where $P^*$ represents the adjoint of a Markov kernel $P$ (see Subsection \ref{subsec:operator}). In this version, the coordinates are updated first in increasing order, from $1$ to $d$, and then in decreasing order, from $d$ to $1$.

	\subsection{Assumptions on the target distribution} \label{subsec:assumptions}

	In the following, let $\abs{\cdot}$ denote the standard Euclidean norm on $\R^d$.
	Unless stated otherwise, all expectations and variances are taken with respect to the target distribution $\pi$.
	We say that a function \mbox{$f: \R^d \to \R$} is locally Lipschitz if its restriction to any open ball in $\R^d$ is Lipschitz continuous.
	With these conventions in mind, we now consider two functional inequalities for $\pi$, namely the Poincar\'e and log-Sobolev inequalities. 
	
	\begin{defn}
		\label{PI}
		Let $q \geq 1$. The distribution $\pi$ is said to satisfy an $L^q$-Poincar\'e inequality (\po{q}) with constant $\cpo{q} > 0$ if for any locally Lipschitz function $f: \R^d \to \R$,
		\begin{equation} \label{PIid}
			\E\left[\abs{f - \E [ f ]}^q \right]\leq \frac{1}{\cpo{q}} \cdot \E \left[ \abs{ \nabla f }^q\right].
		\end{equation}
		We take \cpo{q} to be the largest constant for which this inequality holds.
	\end{defn}
	
	\begin{defn}
		\label{LSI}
		Let $q \geq 1$. The distribution $\pi$ is said to satisfy an $L^q$-log-Sobolev \hbox{inequality} (\ls{q}) with constant $\cls{q} > 0$ if for any locally Lipschitz function \mbox{$f: \R^d\to\R$},
		\begin{equation} \label{LSIid}
			\E\left[\abs{f }^q \cdot \log\abs{f }^q \right] - \E\left[\abs{f }^q \right] \cdot \log \E\left[ \abs{f}^q \right] \leq \frac{1}{\cls{q}} \cdot \E \left[ \abs{\nabla f}^q\right].
		\end{equation}
		As before,	we take \cls{q} to be the largest constant for which this inequality holds.
	\end{defn}
	
	These functional inequalities extend the classical assumption of (strong) log-concavity for $\pi$. Specifically, \cite{bobkov1999isoperimetric} proved that \po{1} is satisfied by all log-concave distributions, and \citet{bakry2006diffusions} showed that \ls{2} holds for all strongly log-concave distributions. As shown in \citet[Proposition 6.2]{bobkov2005entropy}, \ls{q} is also preserved by bounded perturbation \`a la \citet{holley1987logarithmic}.
	We recall that a distribution $\nu$ is said to be a bounded perturbation of another distribution $\mu$	if there exist constants $0 < c_1 \leq c_2 < \infty$ such that the Radon-Nikodym derivative satisfies $c_1 \leq \mathd \nu / \mathd \mu \leq c_2$.
	More generally, \citet{ledoux2006concentration} proved that \po{2} and \ls{2} respectively imply sub-exponential and sub-Gaussian tail bounds for the distribution $\pi$. Poincar\'e inequalities have been established for a broad class of distributions satisfying a Lyapunov condition \citep{bakry2008simple} and log-Sobolev inequalities have also been proven for mixture distributions \citep{chen2021dimension}.
	
	Furthermore, if $\pi$ satisfies \ls{q} with constant $\cls{q}$, then \citet[Theorem 2.1]{bobkov2005entropy} showed that $\pi$ satisfies \po{q} with constant $\cpo{q} \geq 4\cls{q}/ \log 2$. In addition, it can be shown that these functional inequalities form a hierarchy: if $\pi$ satisfies \po{q}, then $\pi$ satisfies \po{r} for any $1 \leq q \leq r$ \citep[Proposition 2.5]{milman2009role}. An analogous result  holds for log-Sobolev inequalities, as formalized in the following proposition.
	
	\begin{prop} \label{LSIhierarchy}
		Let $1 \leq q \leq r$. Suppose that the distribution $\pi$ satisfies \ls{q} with constant \cls{q}. Then, $\pi$ also satisfies \ls{r} with constant
		$$
		\frac{5}{128} \left(\frac{4}{105}\right)^{r/q}
		\left(\frac{q}{r}\right)^r \cls{q}^{r/q} \leq \cls{r}.
		$$
	\end{prop}
	
	To the best of our knowledge, Proposition~\ref{LSIhierarchy} does not appear explicitly in the literature. For completeness, we include a proof in Appendix~\ref{sec:appendixproofs}, which may be of independent interest.
	
	We now introduce one final regularity assumption on the conditional distributions of~$\pi$. 
	To formalize this, recall that the total variation (TV) distance between two probability distributions $\mu$ and $\nu$ supported on $\R^d$ is given by
	$$
	\norm{\mu- \nu}_\tv = \sup \left\{ \abs{\mu (S) - \nu (S)} : S \subseteq \R^d~ \mathrm{ measurable} \right\}.
	$$
	
	\begin{defn} \label{def:tvcontinuous}
		Let $M > 0$ and $\beta \in (0,1]$. The conditional distributions of $\pi$ are said to be $(M, \beta)$-TV continuous if
		$$
		\norm{\pi( \cdot \mid x_{-k}) - \pi( \cdot \mid y_{-k})}_\tv \leq M \abs{x_{-k} - y_{-k}}^\beta \qquad \forall k \in [d], \,\forall x, y \in \R^d.
		$$
	\end{defn}

	This type of continuity has been shown to hold for posterior distributions in Bayesian models, under assumptions such as Lipschitz continuity of the likelihood and sufficient integrability of the prior \citep{dolera2020uniform, dolera2023lipschitz}.
	Moreover, in Subsection \ref{subsection:potentials}, we will show how to verify that the conditional distributions of $\pi$ are TV continuous under standard conditions on the potential function $U = - \log \pi$.

	\section{Convergence analysis of Markov chains} \label{sec:markovconv}
	
	\subsection{Operator representation and spectral gap} \label{subsec:operator}
	Let
	$$
	L^2_\pi = \left\{ f : \R^d \to \R : \int f(x)^2 \pi (\mathd x) < \infty\right\}
	$$
	denote the Hilbert space of square $\pi$-integrable functions. For all $f, g \in L^2_\pi$, the scalar product associated to this space is
	$$
	\scp{f}{g} = \int f(x) g(x) \pi (\mathd x),
	$$
	with corresponding $L^2_\pi$-norm $\norm{f}_2 = \sqrt{\scp{f}{f}}$. Note that a $\pi$-invariant Markov kernel $P$ can be seen as an operator that acts on functions $f \in L^2_\pi$ via
	$$
	Pf(x) = \int f(y) P(x, \mathd y).
	$$
	With this identification, reversibility of $P$ is equivalent to self-adjointness of the associated operator on $L^2_\pi$. In the reversible setting, the Markov kernel $P$ is said to be positive semi-definite if for all $f \in L^2_\pi$, $$\scp{f}{Pf} \geq 0.$$
	In particular, the random-scan Gibbs kernel $P_\mathrm{RS}$ is positive semi-definite. Indeed, each kernel $P_k$ defined in \eqref{kernelPk} satisfies
	$
		P_k f (x) = \E[ f(X) \mid X_{-k} = x_{-k}]
	$
	and is positive semi-definite since
	\begin{align*}
		\scp{f}{P_k f} = \E [ f(X) \cdot \E[ f(X) \mid X_{-k}]] 
		= \E \left[\E\left[ f(X) \mid X_{-k}\right]^2\right] 
		\geq 0.
	\end{align*}
	It follows that
	$$
	\scp{f}{P_\mathrm{RS}f} = \frac{1}{d} \sum_{k=1}^d \scp{f}{P_kf} \geq 0.
	$$
	A common way to ensure positive semi-definiteness is to work with lazy Markov kernels \citep[see, for example,][Subsection 1.b]{lovasz1993random}.
	We say that the kernel $P$ is lazy if
	$$
	P(x, \{ x\}) \geq \frac{1}{2} \qquad \forall x \in \R^d.
	$$
	Any (reversible or non-reversible) Markov kernel can be made lazy as follows.
	Let $I$ denote the identity operator, defined by $I f = f$.
	
	\begin{defn}
		The lazy version of a Markov kernel $P$ is defined as
		$$
		\widetilde P (x, \mathd y) = \frac{1}{2} \delta_x (\mathd y) + \frac{1}{2} P (x,\mathd y).
		$$
		The associated operator is $\widetilde P = (I+P )/2$.
	\end{defn}
	
	In particular, the lazy version of a $\pi$-invariant Markov kernel $P$ is also $\pi$-invariant.
	An important quantity for analyzing the convergence of reversible Markov chains is the spectral gap, which we now recall.

	\begin{defn}
		Consider a $\pi$-invariant Markov chain with transition kernel $P$. Suppose that $P$ is reversible. The spectral gap of $P$ is defined as
		$$
			\lambda_2 (P) = \inf \left\{ \frac{\scp{f}{(I - P) f} }{\norm{f}^2_2}: f \in L^2_\pi,\, \E [f] = 0, \, f \neq 0 \right\}.
		$$
	\end{defn}
	
	When $P$ is reversible, the spectral gap provides a bound on the $L^2_\pi$-convergence rate of the Markov chain: for any $f \in L^2_\pi$ with $\E [f] = 0$, 
	$$
	\norm{Pf}_2 \leq \left(1 - \lambda_2(P)\right) \norm{f}_2.
	$$
	See \cite[Exercise 12.4]{levin2017markov} for a proof.
	For non-reversible chains, various notions of spectral gaps have been proposed; see, for instance, \citet{kontoyiannis2012geometric} for definitions based on suitably weighted Banach spaces, and \citet{chatterjee2023spectral} for finite state spaces.
	Other approaches \citep{fill1991eigenvalue, paulin2015concentration, choi2020metropolis} use reversibilized Markov kernels  to derive similar bounds on the convergence rate.
	However, the sharpness of such bounds in the non-reversible setting remains unclear, and a detailed investigation lies beyond the scope of this work.

	\subsection{Mixing times of Markov chains}
	\label{subsec:mixing}
	
	We begin by recalling key definitions and concepts related to the mixing of Markov chains. 
	Let $\mu$ denote the initial distribution of the Markov chain associated to a kernel $P$. After $k\geq 0$ steps, the distribution of the chain is given by $\mu P^k$, where
	$$
	\mu P^k (\mathd x) = \int P^k(y,\mathd x) \mu (\mathd y).
	$$
	It is often assumed that the initial distribution is ``not too far'' from the invariant distribution, a condition referred to as a warm start.
	
	\begin{defn} \label{warmstart}
		Let $\Omega >0 $. A $\pi$-invariant Markov chain with initial distribution $\mu$ is said to have a $\Omega$-warm start if
		$$
		\sup \left\{ \frac{\mu(S)}{\pi (S)} : S \subseteq \R^d~ \mathrm{ measurable},\,  \pi (S) > 0 \right\} \leq \Omega.
		$$
	\end{defn}

	Note that the constant $\Omega$ can grow exponentially with the dimension $d$ even in fairly innocent examples.
	For instance, this occurs when the initial distribution is a Gaussian centered at the mode of a strongly log-concave target \citep{dwivedi2019log}.
	Nevertheless, when $\pi$ is strongly log-concave, \citet{altschuler2024faster} showed that an initial distribution that is within $\Oc(1)$ 
	$\chi^2$-distance 
	from $\pi$ can be efficiently obtained in $\Oc(\sqrt{d})$ iterations using the underdamped Langevin dynamics. We note, however, that this notion of proximity for the initial distribution is weaker than the warm start in Definition \ref{warmstart}.
	We now define the mixing time of a Markov chain as the minimum number of iterations required for the chain to be ``close'' to its invariant distribution in total variation distance.
	
	\begin{defn}
		Let $\zeta > 0$. The $\zeta$-mixing time of a $\pi$-invariant Markov chain with transition kernel $P$ is defined as
		$$
		\tau (\zeta, P) = \inf\left\{ k \geq 0 : \norm{\mu P^k - \pi}_{\tv} \leq \zeta\right\}.
		$$
	\end{defn}
	
	A common approach to bound the mixing time is to study a quantity known as the conductance of the Markov chain, which is defined as follows.
	
	\begin{defn}
		Consider a $\pi$-invariant Markov chain with transition kernel $P$.
		The conductance of $P$ is defined as
		$$
			\Phi (P) = \inf \left\{\frac{\int_S P(x,S^c) \pi (\mathd x) }{\pi (S)} : S \subseteq \R^d~\mathrm{ measurable},\,  0<\pi (S) \leq \frac{1}{2} \right\}.
		$$
	\end{defn}
	
	In particular, the conductance of a Markov kernel $P$ and that of its lazy version $\widetilde{P}$ are related by
	\begin{equation} \label{condlazy}
		\Phi( \widetilde{P}) = \frac{1}{2} \cdot \Phi(P).
	\end{equation}
	The connection between the mixing time of a Markov chain and its conductance was established by \citet{lovasz1993random}, who extended the earlier work of \citet{jerrum1988conductance} to general state spaces.
	
		\begin{thm} \label{lovsimThm}
			Consider a $\pi$-invariant Markov chain with transition kernel $P$ initialized at a $\Omega$-warm start. Suppose that $P$ is either lazy, or both reversible and positive semi-definite. Then, for all $k \geq 0$,
			$$
			\norm{\mu P^k - \pi}_{\tv} \leq \sqrt{\Omega} \cdot \exp \left(- \frac{k}{2} \cdot \Phi(P)^2 \right).
			$$
		\end{thm}
		\begin{proof}
			If $P$ is lazy, the bound follows from \citet[Corollary 1.5]{lovasz1993random}. If $P$ is reversible and positive semi-definite, the same bound follows from \citet[Lemma 5.3]{lee2024fast}, which is a modified version of the former result.
		\end{proof}
	
	\begin{cor} \label{lovsimCor}
		Let $\zeta > 0$. Under the conditions of Theorem \ref{lovsimThm}, the $\zeta$-mixing time of the Markov chain is bounded by
		$$
		\tau (\zeta, P) \leq \frac{2}{\Phi(P)^2} \cdot \log \left( \frac{\sqrt{\Omega}}{\zeta}\right).
		$$
	\end{cor}

	The problem of upper bounding the mixing time is thus reduced to proving a lower bound on the conductance. The spectral gap of a reversible Markov kernel is also related to its conductance via the following bounds, also known as Cheeger's inequalities \citep{cheeger1970lower}.
		
		\begin{thm}[\citealp{lawler1988bounds}, Theorem 2.1] \label{cheegerinequalties}
			Let $\Phi (P)$ and $\lambda_2 (P)$ denote respectively the conductance and spectral gap of a reversible Markov kernel $P$. Then,
			$$
			\frac{1}{2}\cdot \Phi^2 (P) \leq \lambda_2 (P) \leq \Phi (P).
			$$
		\end{thm}	
		
		Theorem~\ref{cheegerinequalties} applies specifically to the reversible setting and covers the random-scan Gibbs sampler as well as reversibilized forms of the systematic-scan Gibbs sampler, such as the variant with kernel $P_{\mathrm{SS}} P_{\mathrm{SS}}^*$ discussed in Subsection \ref{subsec:gs}.
	
	\subsection{Bounding the conductance} \label{subsection:boundingconductance}
	
	In order to bound the conductance of a $\pi$-invariant Markov kernel $P$, two main components are required: an isoperimetric inequality for the distribution $\pi$, and a close coupling condition for the kernel $P$.
	Following the terminology of \citet{andrieu2024explicit}, we refer to the particular isoperimetric inequalities used here as  ``three-set'' isoperimetric inequalities.
	In what follows, we write $\id: \R \to \R$ for the identity function and use $\sqcup$ to denote the union of disjoint sets.
	
	\begin{defn} \label{3sidef}
		Let $\Upsilon: (0, \infty) \to (0, \infty)$ and $\Psi: (0,1/2) \to (0,\infty)$ be two 	non-decreasing	functions. Additionally, suppose that $\Psi/\id$ is 
		non-increasing.
		The distribution $\pi$ is said to satisfy a three-set isoperimetric inequality with functions $\Upsilon, \Psi$, if for all measurable partitions $S_1 \sqcup S_2 \sqcup S_3 = \R^d$,
		$$
		\pi ( S_3) \geq \Upsilon(\mathcal{D} ( S_1, S_2 )) \cdot \Psi(\min \{ \pi ( S_1), \pi ( S_2) \}),
		$$
		where $\mathcal{D} ( S_1, S_2) = \inf \{ \abs{x-y} : x \in S_1, \, y \in S_2 \}$ denotes the distance between  $S_1$ and $S_2$.
	\end{defn}
	
	When $\pi$ is log-concave (or a bounded perturbation thereof), a three-set inequality can be derived by controlling a quantity known as the isoperimetric profile of $\pi$
	\citetext{\citealp[Theorem 1.2]{milman2009role2}; \citealp[Section 3]{andrieu2024explicit}}.
	Appropriate bounds on the isoperimetric profile are known to be equivalent to $\pi$ satisfying $\po{1}$ and, similarly, $\ls{1}$ \citep{rothaus1985analytic, bobkov1997some, kolesnikov2007modified}.
	In all these cases, $\Upsilon$ is taken to be a multiple of  the identity function.
	However, this approach generally breaks down for non-log-concave $\pi$, as shown in \citet[Subsection 1.2]{milman2009role}.
	
	\begin{defn}
		Let $\delta, \varepsilon > 0$. A Markov kernel $P$ is said to satisfy a $(\delta, \varepsilon)$-close coupling condition if
		$$
		\abs{x-y} \leq \delta \implies \norm{P (x, \cdot) - P (y, \cdot)}_\tv \leq 1 - \varepsilon \qquad \forall x,y \in \R^d.
		$$
	\end{defn}
	
	The close coupling condition ensures that when $x$ and $y$ are ``close'' in Euclidean distance, the probability distributions $P(x, \cdot)$ and $P(y, \cdot)$ are ``close'' in total variation distance. 
	Assuming that $\pi$ satisfies a three-set isoperimetric inequality and $P$ satisfies a close coupling condition, we are able to establish a lower bound on the conductance of the Markov chain.

	\begin{prop} \label{conductanceTHM}
		Consider a $\pi$-invariant Markov kernel $P$. Suppose that $\pi$ satisfies a three-set isoperimetric inequality with functions $\Upsilon, \Psi$, and $P$ satisfies a $(\delta, \varepsilon)$-close coupling condition. Then, for any measurable $S \subseteq \R^d$ with $\pi(S) \leq 1/2$,
		$$
		\int_S P(x,S^c) \pi (\mathd x) \geq \frac{\varepsilon}{4} \cdot \min \left\{  \pi (S ), \Upsilon(\delta) \cdot \Psi\left( \frac{\pi (S )}{2} \right) \right\}.
		$$
	\end{prop}
	
	\begin{cor} \label{conductanceProfileLowerBound}
		Under the conditions of Proposition \ref{conductanceTHM}, the conductance of the Markov kernel $P$ can be bounded by
		$$
		\Phi (P)\geq \frac{\varepsilon \cdot \Psi (1/4)}{2}  \cdot  \min \left\{\frac{1}{2 \cdot \Psi (1/4)}, \Upsilon (\delta)\right\}.
		$$
	\end{cor}
	
	The proof of Proposition \ref{conductanceTHM} can be found in Appendix \ref{sec:appendixproofs} and builds on several previous works; see \citet[Lemma 15]
	{andrieu2024explicit} and the survey by \citet{vempala2005geometric} for further details. 
	Notably, Proposition \ref{conductanceTHM} and Corollary \ref{conductanceProfileLowerBound} do not require the Markov kernel $P$ to be reversible.
	Conventionally, close coupling conditions are established with the parameter $\delta$ scaling inversely with the dimension $d$ and the parameter~$\varepsilon$ being of constant order.
	Under this regime, Corollary~\ref{conductanceProfileLowerBound} yields
	$$
	\Phi(P)\gtrsim \min \{1, \Upsilon (\delta)\} = \Upsilon (\delta)
	$$
	for sufficiently large $d$.

	\section{Main results} \label{sec:mainresults}
	
	\subsection{Three-set isoperimetric inequalities via functional inequalities} \label{sec:3sifi}
	
	We now establish three-set isoperimetric inequalities for distributions $\pi$ satisfying an \hbox{$L^q$-Poincar\'e} or $L^q$-log-Sobolev inequality for any $q \geq 1$.
	As noted in Subsection \ref{subsection:boundingconductance}, traditional methods based on the isoperimetric profile cease to be effective when $\pi$ is not log-concave.
	To overcome these challenges, we take a more direct approach to derive  three-set isoperimetric inequalities for distributions satisfying \po{q} or \ls{q} with $q\geq1$. In these settings, $\Upsilon$ is no longer restricted to be proportional to the identity function. In what follows, denote the indicator function of a measurable set $S \subseteq \R^d$ as $\mathbbm{1}_S$.  Additionally, define the $L^q_\pi$-norm of a function $f$ as
	$$
	\norm{f}_q = \left( \int \abs{f(x)}^q \pi(\mathd x)\right)^{1/q}.
	$$
	
	\begin{prop} \label{3siPI}
		Let $q \geq 1$. If the distribution $\pi$ satisfies \po{q} with constant $\cpo{q}$, then $\pi$ satisfies a three-set isoperimetric inequality with functions
		\begin{gather*}
			\Upsilon(t) = \frac{\cpo{q} t^q}{2^{2q-1} (1 + 2^q  \cpo{q} t^q)}, \qquad
			\Psi(t) = t.
		\end{gather*}
	\end{prop}
	
	\begin{proof}
		Let $S_1 \sqcup S_2 \sqcup S_3 = \R^d$ be a partition of space with $\mathcal{D}(S_1, S_2) > 0$. Inspired by \citet[Proposition 1.7]{ledoux2001concentration}, define the function
		$$
		f(x) = \max \left\{ 0, 1 - \frac{\mathcal{D}(x, S_1)}{\mathcal{D}(S_1, S_2)} \right\}.
		$$
		Observe that the map $x \mapsto \mathcal{D} (x, S_1)$ is Lipschitz continuous with constant 1.
		Indeed, for any $\xi >0$, let $s_\xi \in S_1$ be such that $\abs{y-s_\xi} \leq \mathcal{D} (y, S_1) + \xi$. Then,
		\begin{align*}
			\mathcal{D} (x, S_1) &\leq \abs{x - s_\xi} \\
			&\leq \abs{x-y} + \abs{y-s_\xi} \\
			&\leq \abs{x-y} + \mathcal{D} (y, S_1) + \xi,
		\end{align*}
		so that $\mathcal{D} (x, S_1) - \mathcal{D} (y, S_1) \leq \abs{x-y} + \xi$.
		Swapping the roles of $x$ and $y$ and letting $\xi \to 0$, we obtain  $\abs{\mathcal{D} (x, S_1) - \mathcal{D} (y, S_1)} \leq \abs{x-y}$.
		Hence, the function $f$ is Lipschitz continuous on $\R^d$.
		By Rademacher's theorem, $f$ is differentiable almost everywhere and satisfies 
		$$
		\abs{\nabla f(x)} \leq \frac{1}{\mathcal{D} (S_1, S_2)}
		$$ 
		for almost every $x \in \R^d$.
		Moreover, the function $f$ takes the constant value 1 on $S_1$, vanishes on $S_2$, and satisfies $f(x) \in [0,1]$ for all $x \in S_3$.
		Applying \eqref{PIid} to $f$, we obtain
		\begin{equation} \label{3siPI:direct}
			\norm{f - \mathbb{E}[f]}_q^q = \mathbb{E} [ \abs{f - \mathbb{E}[f]}^q ] \leq \frac{1}{\cpo{q}} \cdot \mathbb{E} [ \abs{\nabla f}^q ] \leq \frac{\pi(S_3)}{\cpo{q} \cdot \mathcal{D}(S_1, S_2)^q}.
		\end{equation}
		We now compare $\mathbb{E}[\abs{f - \mathbb{E}[f]}^q]$ to the corresponding quantity for the indicator function~$\mathbbm{1}_{S_1}$, namely $\mathbb{E}\left[\abs{\mathbbm{1}_{S_1} - \mathbb{E}[\mathbbm{1}_{S_1}]}^q\right]$. 
		By the triangle inequality and Jensen's inequality, we get
		\begin{align*}
			\mathbb{E}[\abs{\mathbbm{1}_{S_1} - \mathbb{E}[\mathbbm{1}_{S_1}]}^q]^{1/q} &=\norm{\mathbbm{1}_{S_1} - \mathbb{E}[\mathbbm{1}_{S_1}]}_q \\
			&\leq \norm{f - \mathbbm{1}_{S_1}}_q + \norm{f - \mathbb{E}[f]}_q + \abs{\mathbb{E}[f] - \mathbb{E}[\mathbbm{1}_{S_1}]} \\
			&\leq 2 \norm{f - \mathbbm{1}_{S_1}}_q + \norm{f - \mathbb{E}[f]}_q.
		\end{align*}
		Using the inequality $(a + b)^q \leq 2^{q-1} (a^q + b^q)$ for $a, b \geq 0$ and $q \geq 1$, it follows that
		\begin{equation} \label{3siPI:1} 
			\mathbb{E}[\abs{\mathbbm{1}_{S_1} - \mathbb{E}[\mathbbm{1}_{S_1}]}^q] \leq 2^{q-1} \left( 2^q \norm{f - \mathbbm{1}_{S_1}}_q^q + \norm{f - \mathbb{E}[f]}_q^q \right).
		\end{equation}
		Since $f - \mathbbm{1}_{S_1}$ vanishes on $S_1 \sqcup S_2$ and takes values in $[0, 1]$ on $S_3$, we have
		\begin{equation} \label{3siPI:2} 
			\norm{f - \mathbbm{1}_{S_1}}_q^q = \mathbb{E} [ \abs{f - \mathbbm{1}_{S_1}}^q ] \leq \pi(S_3).
		\end{equation}
		Inserting \eqref{3siPI:direct} and \eqref{3siPI:2} into \eqref{3siPI:1}, we get
		$$
		\mathbb{E}[\abs{\mathbbm{1}_{S_1} - \mathbb{E}[\mathbbm{1}_{S_1}]}^q] \leq 2^{q-1} \left( 2^q + \frac{1}{\cpo{q} \cdot \mathcal{D}(S_1, S_2)^q} \right) \pi(S_3).
		$$
		We now derive a lower bound for $\mathbb{E}[\abs{\mathbbm{1}_{S_1} - \mathbb{E}[\mathbbm{1}_{S_1}]}^q]$:
		\begin{align*}
			\mathbb{E}[\abs{\mathbbm{1}_{S_1} - \mathbb{E}[\mathbbm{1}_{S_1}]}^q] &= (1 - \pi(S_1))^q \cdot \pi(S_1) + \pi(S_1)^q (1 - \pi(S_1)) \\
			&= (1 - \pi(S_1)) \cdot \pi(S_1) \left[\left(1 - \pi(S_1)\right)^{q-1} + \pi(S_1)^{q-1}\right] \\
			&\geq \frac{1}{2^{q-1}} (1 - \pi(S_1)) \cdot\pi(S_1) \\
			&\geq \frac{1}{2^q} \cdot \min\{\pi(S_1), 1 - \pi(S_1)\} \\
			&\geq \frac{1}{2^q} \cdot \min\{\pi(S_1), \pi(S_2)\},
		\end{align*}
		where we have used the inequality $1/2 \cdot \min\{x, 1-x\} \leq x(1-x)$ for $x \in [0, 1]$.
		Combining these bounds, we obtain
		$$
		\min\{\pi(S_1), \pi(S_2)\} \leq 2^{2q-1} \left( 2^q + \frac{1}{\cpo{q} \cdot \mathcal{D}(S_1, S_2)^q} \right) \pi(S_3),
		$$
		or equivalently,
		$$ 
		\pi(S_3) \geq \frac{\mathcal{D}(S_1, S_2)^q}{2^{2q-1} \left( \cpo{q}^{-1} + 2^q \cdot \mathcal{D}(S_1, S_2)^q \right)} \cdot \min\{\pi(S_1), \pi(S_2)\}.
		$$
		The $L^q$-Poincar\'e inequality thus implies a three-set isoperimetric inequality
		$$
		\pi (S_3) \geq \Upsilon (\mathcal{D}(S_1, S_2)) \cdot \min\{\pi(S_1), \pi(S_2)\},
		$$
		where
		$$
		\Upsilon(t) = \frac{\cpo{q} t^q}{2^{2q-1}(1 + 2^q \cpo{q} t^q)}.
		$$
		
	\end{proof}
	
	To derive a three-set isoperimetric inequality corresponding to the $L^q$-log-Sobolev inequality, it is useful to analyze the properties of the function $x \mapsto x \cdot \log (1/x)$. The proof of the following lemma is provided in Appendix \ref{sec:appendixproofs}. Throughout, we adopt the convention that $0 \cdot \log (0)  = 0$.
	
	\begin{lem} \label{3siLSIlem}
		For $x \in [0,1]$, consider the function  
		$
		\rho(x) = x \cdot \log (1/x)
		$.
		Then:
		\begin{itemize}
			\item The function $\rho(x)$ is positive and achieves its maximum at $x = e^{-1}$.
			\item For all $x,h \in [0,1]$ such that $x + h \leq 1$,
			$
			\rho(x+h) \geq \rho(x) - h
			$.
		\end{itemize}
	\end{lem}

	\begin{prop} \label{3siLSI}
		Let $q \geq 1$. If the distribution $\pi$ satisfies \ls{q}  with constant \cls{q}, then $\pi$ satisfies a three-set isoperimetric inequality with functions
		\begin{gather*}
			\Upsilon(t) = \frac{\cls{q} t^q}{1 + (1 + e^{-1}) \cls{q} t^q}, \qquad	\Psi(t) = \frac{t}{2} \cdot \log \left( \frac{2}{t}\right).
		\end{gather*}
	\end{prop}
	
	\begin{proof}
		Let $S_1 \sqcup S_2 \sqcup S_3 = \R^d$ be a partition of space with $\mathcal{D}(S_1, S_2) > 0$. Similarly to Proposition \ref{3siPI}, define the function
		$$
		f(x) = \max \left\{ 0, 1 - \frac{\mathcal{D}(x, S_1)}{\mathcal{D}(S_1, S_2)} \right\}.
		$$
		As established above, the function $f$ is Lipschitz continuous with $\abs{\nabla f(x)} \leq \mathcal{D} (S_1, S_2)^{-1}$
		for almost every $x \in \R^d$. 
		Moreover,  $f$ takes the constant value 1 on $S_1$, vanishes on $S_2$, and satisfies $f(x) \in [0,1]$ for all $x \in S_3$.
		Applying \eqref{LSIid} to $f$, we obtain
		\begin{equation} \label{3siLSI:direct}
			\E[\abs{f }^q \cdot \log\abs{f }^q ] - \E[\abs{f }^q ] \cdot \log \E[ \abs{f}^q ] \leq \frac{1}{\cls{q}} \cdot \mathbb{E} [ \abs{\nabla f(x)}^q ] \leq \frac{\pi (S_3)}{\cls{q} \cdot \mathcal{D} (S_1, S_2)^q}.
		\end{equation}
		For $x \in [0,1]$, consider the function  
		$
		\rho(x) = x \cdot \log (1/x)
		$.
		We can rewrite
		$$
		\E[\abs{f }^q \cdot \log\abs{f }^q ] = {\int_{S_3}} \abs{f(x)}^q \cdot \log (\abs{f(x)}^q) \pi(\mathd x) = - {\int_{S_3}} \rho (\abs{f (x)}^q) \pi (\mathd x)
		$$
		and
		\begin{align*}
			\E[\abs{f }^q ] \cdot \log \E[ \abs{f}^q ] &= \left( \pi (S_1) + \int_{S_3} \abs{f(x)}^q \pi(\mathd x) \right) \cdot \log \left( \pi (S_1) + \int_{S_3} \abs{f(x)}^q \pi(\mathd x) \right) \\
			&= - \rho \left( \pi (S_1) + \int_{S_3} \abs{f(x)}^q \pi(\mathd x) \right) .
		\end{align*}
		By the first part of Lemma \ref{3siLSIlem},
		\begin{equation} \label{3siLSI:1}
				\E[\abs{f }^q \cdot \log\abs{f }^q ] \geq - e^{-1} \cdot \pi(S_3).
		\end{equation}
		Noting that $\E[\abs{f}^q]\leq 1$, the second part of Lemma \ref{3siLSIlem} gives
		\begin{align*}
			\rho\left( \pi (S_1) + \int_{S_3} \abs{f(x)}^q \pi(\mathd x) \right) 		& \geq \pi	(S_1) \cdot \log \left( \frac{1}{\pi (S_1)} \right) - \int_{S_3} \abs{f(x)}^q \pi(\mathd x) \\
			&\geq  \pi	(S_1)\cdot \log \left( \frac{1}{\pi (S_1)} \right) - \pi(S_3),
		\end{align*}
		and thus
		\begin{equation} \label{3siLSI:2}
			-\E[\abs{f }^q ] \cdot \log \E[ \abs{f}^q ] \geq \pi	(S_1)\cdot \log \left( \frac{1}{\pi (S_1)} \right) - \pi(S_3).
		\end{equation}
		Inserting \eqref{3siLSI:1} and \eqref{3siLSI:2} into \eqref{3siLSI:direct}, we obtain
		$$
		\frac{\pi (S_3)}{\cls{q} \cdot \mathcal{D} (S_1, S_2)^q} \geq \pi(S_1)\cdot \log \left( \frac{1}{\pi (S_1)} \right) - (1 + e^{-1}) \pi(S_3),
		$$
		or equivalently,
		$$
		\pi (S_3) 
		\geq \frac{\mathcal{D} (S_1, S_2)^q}{\cls{q}^{-1} + (1 +e^{- 1})\cdot  \mathcal{D}(S_1, S_2)^q} \cdot \pi (S_1) \cdot \log \left( \frac{1}{\pi (S_1)} \right).
		$$
		By swapping $S_1$ with $S_2$ in the definition of the function $f$, an identical argument yields
		$$
		\pi (S_3) 
		\geq \frac{\mathcal{D} (S_1, S_2)^q}{\cls{q}^{-1} + (1 +e^{- 1}) \cdot  \mathcal{D}(S_1, S_2)^q}\cdot  \pi (S_2)  \cdot \log \left( \frac{1}{\pi (S_2)} \right).
		$$
		Combing these results, we have 
		\begin{align*}
			\pi (S_3) 
			&\geq \frac{\mathcal{D} (S_1, S_2)^q}{\cls{q}^{-1} + (1 +e^{- 1})\cdot  \mathcal{D}(S_1, S_2)^q}\cdot \max \left\{ \pi (S_1)\cdot \log \left( \frac{1}{\pi (S_1)} \right), \pi (S_2) \cdot\log \left( \frac{1}{\pi (S_2)} \right) \right\} \\
			&\geq \frac{\mathcal{D} (S_1, S_2)^q}{\cls{q}^{-1} + (1 +e^{- 1})\cdot  \mathcal{D}(S_1, S_2)^q} \cdot \min\{\pi(S_1), \pi(S_2)\}\cdot \log \left( \frac{1}{\min\{\pi(S_1), \pi(S_2)\} } \right) \\
			&\geq \frac{\mathcal{D} (S_1, S_2)^q}{\cls{q}^{-1} + (1 +e^{- 1})\cdot  \mathcal{D}(S_1, S_2)^q} \cdot \frac{\min\{\pi(S_1), \pi(S_2)\}}{2}\cdot \log \left( \frac{2}{\min\{\pi(S_1), \pi(S_2)\} } \right)
		\end{align*}
		The advantage of this last expression is that the function $x \mapsto x/2 \cdot \log (2/x)$ is increasing everywhere on $(0,1/2)$, contrary to the function $x \mapsto x \cdot \log (1/x)$. The $L^q$-log-Sobolev inequality thus implies a three-set isoperimetric inequality
		$$
		\pi (S_3) \geq \Upsilon (\mathcal{D}(S_1, S_2)) \cdot \Psi (\min\{\pi(S_1), \pi(S_2)\}),
		$$
		where
		$$
		\Upsilon(t) = \frac{\cls{q} t^q}{1 + (1 +e^{- 1}) \cls{q} t^q}, \qquad \Psi(t) = \frac{t}{2} \cdot \log \left( \frac{2}{t}\right).
		$$
	\end{proof}
	
	Notably, the $L^q$-Poincar\'e and $L^q$-log-Sobolev inequalities respectively lead to three-set isoperimetric inequalities with $\Upsilon(t) \sim \cpo{q}t^q/2^{2q}$ and $\Upsilon(t) \sim \cls{q}t^q$ as $t\to 0$.
	Here, the sign~$\sim$ means that the two quantities are equivalent as $t \to 0$, in the sense that their ratio converges to a finite constant.
	This regime is particularly relevant in high-dimensional settings, where $t$ typically scales inversely with the dimension, as per the discussion following Corollary \ref{conductanceProfileLowerBound}.
	In contrast, when $\pi$ is log-concave, $\Upsilon(t)\sim t$ as $t\to 0$ \citep[Subsection 3.2]{andrieu2024explicit}.
	
	In many settings, one can only verify a close coupling condition on a high probability region $K \subsetneq \R^d$, which leads to lower bounds on the $s$-conductance (or restricted conductance) rather than  the usual conductance; see, for example, \citet{wu2022minimax, chen2023simple, chewi2021optimal, dwivedi2019log, mangoubi2019nonconvex, mou2022efficient, chen2020fast}.
	In certain cases, obtaining such bounds requires working with a three-set isoperimetric inequality for the distribution $\pi$ truncated to~$K$.
	For log-concave distributions, it is known that truncation preserves three-set isoperimetric inequalities; see \citet[Lemma~15]{chen2020fast} and \citet[Proof of Lemma~6]{dwivedi2019log}, both relying on \citet[Theorem 4.4]{cousins2014cubic}.
	Beyond the log-concave setting, however, it is unclear whether three-set isoperimetric inequalities are preserved under truncation. Extending the three-set isoperimetric inequalities associated with global Poincar\'e or log-Sobolev inequalities to truncated distributions appears delicate and is left for future work.
	
	\subsection{Close coupling conditions for the Gibbs kernels}\label{sec:ccss}
	
	In what follows, let $d\geq 2$. Here, we establish close coupling conditions for the systematic-scan and random-scan Gibbs kernels, assuming  that the conditional distributions of the target distribution $\pi$ are \hbox{TV continuous} (as per Definition \ref{def:tvcontinuous}). 
	The following results rely on sequential coupling arguments; see \citet{roberts2013note} for a formal construction of such couplings, and \citet{marton2004measure, wang2014convergence} for related coupling techniques in Wasserstein distance.
	We begin with the systematic-scan Gibbs sampler.

	\begin{prop} \label{ssgibbsccc}
		Suppose that the conditional distributions of $\pi$ are $(M,\beta)$-TV continuous.
		Then,
		$$
		\norm{P_{\mathrm{SS}} (x, \cdot) - P_{\mathrm{SS}} (y, \cdot)}_\tv \leq 1 - \left( 1 - M \abs{x-y}^\beta\right)^d \qquad \forall x,y \in \R^d.
		$$
		Hence, $P_{\mathrm{SS}}$ satisfies a $(\delta, \varepsilon)$-close coupling condition with
		$$
		\delta  = \frac{1}{(Md)^{1/\beta}}, \qquad \varepsilon = \frac{1}{4}.
		$$
	\end{prop}

	\begin{proof}
		The coupling inequality for the total variation distance can be expressed as
		$$
		\norm{P_{\mathrm{SS}} (x, \cdot) - P_{\mathrm{SS}} (y, \cdot)}_\tv = \inf \left\{ \Pb (U \neq V ) : U \sim P_{\mathrm{SS}} (x, \cdot),  V \sim  P_{\mathrm{SS}} (y, \cdot) \right\}.
		$$
		To construct a candidate coupling of \( P_{\mathrm{SS}}(x, \cdot) \) and \( P_{\mathrm{SS}}(y, \cdot) \), we introduce intermediate states \( U^{(k)} \) and \( V^{(k)} \), for \( k = 0, \dots, d \), which represent the configurations after updating the first \( k \) coordinates of two chains initialized at \( x \) and \( y \), respectively. The coupling proceeds as follows:
		
		\begin{itemize}
			\item Initialize $U^{\left(0\right)} = x$ and $V^{\left(0\right)} = y$.
			\item For $k=1, \dots, d$, draw a maximal coupling
			$$
			W \sim \pi \left( \cdot \mid U^{\left(k-1\right)}_{-k}\right), \qquad W' \sim \pi \left( \cdot \mid V^{\left(k-1\right)}_{-k}\right),
			$$
			such that
			\begin{equation} \label{sscoupling}
				\Pb \left(W \neq W'\mid U^{\left(k-1\right)}_{-k}, V^{\left(k-1\right)}_{-k} \right) = \norm{ \pi \left( \cdot \mid U^{\left(k-1\right)}_{-k}\right) -  \pi \left( \cdot \mid V^{\left(k-1\right)}_{-k}\right)}_\tv.
			\end{equation}
			See \citet[Theorem 5.2]{lindvall2002lectures} for an explicit construction of a maximal coupling. Note that both sides of \eqref{sscoupling} are random quantities depending on $U^{\left(k-1\right)}$, $V^{\left(k-1\right)}$.
			\item Update the variables by setting
			$$
			U^{(k)} = \left(U^{(k-1)}_{-k}, W\right), \qquad V^{(k)} = \left(V^{(k-1)}_{-k}, W'\right).
			$$
			In words, only the $k$'th coordinate is updated when transitioning from $U^{(k-1)}$ to $U^{(k)}$ (and from $V^{(k-1)}$ to $V^{(k)}$). The notation used here for indexing $U^{(k)}$ and $V^{(k)}$ follows the convention introduced in \eqref{notationindex}.
			\item Finally, set $U = U^{\left(d\right)}$, $V = V^{\left(d\right)}$.
		\end{itemize}
		By construction,  $U \sim P_{\mathrm{SS}}(x, \cdot)$ and $V \sim P_{\mathrm{SS}} (y, \cdot)$. Moreover, observe that for all $k \in [d]$, 
		$
		U^{\left(k\right)} = \left( U_{1:k}, x_{k+1:d} \right)$ and  $V^{\left(k\right)} = \left( V_{1:k}, y_{k+1:d} \right)
		$.
		Using this, \eqref{sscoupling} can be rewritten as
		\begin{equation} \label{sscoupling2}
			\Pb (U_k \neq V_k \mid  U_{1:k-1}, V_{1:k-1} ) = \norm{ \pi \left( \cdot \mid U_{1:k-1}, x_{k+1:d}\right) - \pi \left( \cdot \mid V_{1:k-1}, y_{k+1:d}\right)}_\tv.
		\end{equation}
		Conditional on the event $\left\{ U_{1:k-1} = V_{1:k-1}\right\}$, we have
		$$
		\abs{\left( U_{1:k-1}, x_{k+1:d}\right) - \left( V_{1:k-1}, y_{k+1:d}\right)} = \abs{x_{k+1:d} - y_{k+1:d}}.
		$$
		Thus,  by \eqref{sscoupling2} and the assumption that the conditional distributions of \(\pi\) are \((M, \beta)\)-TV continuous, it follows that
		\begin{align*}
			\Pb &(U_k \neq V_k \mid U_{1:k-1} = V_{1:k-1}) \\
			&=\E_{U_{1:k-1}, V_{1:k-1}} \left[ \norm{ \pi \left( \cdot \mid U_{1:k-1}, x_{k+1:d}\right) - \pi \left( \cdot \mid V_{1:k-1}, y_{k+1:d}\right)}_\tv \mid  U_{1:k-1} = V_{1:k-1} \right] \\
			&\leq M\abs{x_{k+1:d} - y_{k+1:d}}^\beta,
		\end{align*}
		where the expectation is taken over the joint distribution of $U_{1:k-1}, V_{1:k-1}$.
		We can now bound the total variation distance between $P_{\mathrm{SS}} (x, \cdot)$ and $P_{\mathrm{SS}} (y, \cdot)$ as follows:
		\begin{align*}
			\norm{P_{\mathrm{SS}} (x, \cdot) - P_{\mathrm{SS}} (y, \cdot) }_\tv 
			&\leq \Pb (U \neq V) \\ 
			&= 1 - \Pb (U = V) \\
			&= 1 - \Pb \left( \bigcap_{k=1}^d \left\{ U_k = V_k \right\} \right) \\
			&= 1 - \prod_{k=1}^d \Pb (U_k = V_k \mid U_{1:k-1} = V_{1:k-1}) \\
			&= 1 - \prod_{k=1}^d \left[1-\Pb (U_k \neq V_k \mid U_{1:k-1} = V_{1:k-1})\right] \\
			&\leq 1 - \prod_{k=1}^d \left(1-  M \abs{x_{k+1:d} - y_{k+1:d}}^\beta\right) \\
			&\leq 1 - \left(1 -  M \abs{x - y}^\beta\right)^d.
		\end{align*}
	\end{proof}
	
	We now turn our attention towards obtaining a close coupling for the random-scan Gibbs sampler. However, one issue immediately becomes apparent: when $x,y \in \R^d$  do not share a coordinate (i.e., when $x-y$ has a non-zero component in each coordinate direction), the supports of $P_{\mathrm{RS}}(x, \cdot)$ and $P_{\mathrm{RS}} (y, \cdot)$ are disjoint. In this case,
	$$
	\norm{P_{\mathrm{RS}} (x, \cdot) - P_{\mathrm{RS}} (y, \cdot)}_\tv = 1.
	$$	
	To circumvent this issue, we adopt a similar approach to \citet{narayanan2022mixing} by establishing a close-coupling condition for the iterated kernel  $P_{\mathrm{RS}}^N$. Notably, when $N \geq d$, $P_{\mathrm{RS}}^N (x, \cdot)$ is globally supported  for all $x \in \R^d$. It turns out that $N \sim d \cdot \log d$ as $d \to \infty$ is indeed the correct scaling in order to get a non-degenerate close coupling condition for the iterated random-scan Gibbs kernel, as formalized in the next lemma. Let $\lceil \cdot \rceil$ denote the ceiling function. The proof can be found in Appendix \ref{sec:appendixproofs}.
	
	\begin{lem} \label{couponlem}
		Let
		$
		I_1, \dots, I_{N} \sim \mathrm{Unif} [d]
		$ be independent and identically distributed.
		Denote by $\Xi$ the set of unique values attained by $I_1, \dots, I_{N}$. If $N =  \lceil 2 d \cdot \log d\rceil$, then
		$$
		\Pb ( \Xi = [d]) \geq \frac{1}{2}.
		$$
	\end{lem}

	Lemma \ref{couponlem} is closely related to the classic coupon collector problem in probability \citep{erdos1961classical}. Using the same notation, observe that when $N=d$, 
	$$
	\Pb  ( \Xi =  [d]) = \frac{d!}{d^d} \sim \frac{\sqrt{d}}{e^d}
	$$
	as $d \to \infty$. This follows from Stirling's formula. The probability of observing all $d$ unique values in exactly $d$ draws thus decreases exponentially with the dimension $d$.
	
	\begin{prop} \label{rsgibbsccc}
		Suppose that the conditional distributions of $\pi$ are $(M,\beta)$-TV continuous. Let  $N_d =  \lceil 2 d \cdot \log d \rceil$. Then,
		$$
		\norm{P_{\mathrm{RS}}^{N_d} (x, \cdot) - P_{\mathrm{RS}}^{N_d} (y, \cdot)}_\tv \leq 1 - \frac{1}{2}  \left( 1 - M \abs{x-y}^\beta\right)^{N_d} \qquad \forall x,y \in \R^d.
		$$
		Hence, $P_{\mathrm{RS}}^{N_d}$ satisfies a $(\delta, \varepsilon)$-close coupling condition with
		$$
		\delta  = \frac{1}{(MN_d)^{1/\beta}}, \qquad \varepsilon =\frac{1}{8}.
		$$
		In particular, $\delta \geq (3Md\cdot \log d)^{-1/\beta}$.
	\end{prop}
	
	\begin{proof}
		Similarly to Theorem \ref{ssgibbsccc}, we construct a sequential coupling of $P_{\mathrm{RS}}^{N_d} (x, \cdot)$ and $P_{\mathrm{RS}}^{N_d} (y, \cdot)$ as follows:
		\begin{itemize}
			\item Draw
			$
			I_1, \dots, I_{N_d} \sim \mathrm{Unif} [d]
			$
			independent and identically distributed.
			\item Initialize $U^{\left(0\right)} = x$ and $V^{\left(0\right)} = y$.
			\item For $k = 1, \dots, d$, draw a maximal coupling
			$$
			W_k \sim \pi \left( \cdot \mid U^{\left(k-1\right)}_{-{I_k}}\right), \qquad W_k' \sim \pi \left( \cdot \mid V^{\left(k-1\right)}_{-{I_k}}\right),
			$$
			such that
			\begin{equation} \label{rscoupling}
				\Pb \left(W_k \neq W_k'\mid U^{\left(k-1\right)}_{-{I_k}}, V^{\left(k-1\right)}_{-{I_k}} \right) = \norm{ \pi \left( \cdot \mid U^{\left(k-1\right)}_{-{I_k}}\right) -  \pi \left( \cdot \mid V^{\left(k-1\right)}_{-{I_k}}\right)}_\tv.
			\end{equation}
			Note that both sides of \eqref{rscoupling} are random quantities depending on $U^{\left(k-1\right)}$, $V^{\left(k-1\right)}, I_k$.
			\item Update the variables by setting
			$$
			U^{(k)} = \left(U^{(k-1)}_{-I_k}, W\right), \qquad V^{(k)} = \left(V^{(k-1)}_{-I_k}, W'\right).
			$$
			In words, only the $I_k$'th coordinate is updated when transitioning from $U^{(k-1)}$ to $U^{(k)}$ (and from $V^{(k-1)}$ to $V^{(k)}$). The notation used here for indexing $U^{(k)}$ and $V^{(k)}$ follows the convention introduced in \eqref{notationindex}.
			\item Finally, set $U = U^{(N_d)}$, $V = V^{\left(N_d\right)}$.
		\end{itemize}
		By construction, $U \sim P_{\mathrm{RS}}^{N_d} (x, \cdot)$ and $V \sim P_{\mathrm{RS}}^{N_d} (y, \cdot)$. 
		As before, denote by $\Xi$ the set of unique values attained by $I_1, \dots, I_{N_d}$. When $x$ and $y$ differ in at least two coordinates, the event $\left\{U = V\right\}$ has a non-zero probability of occurring only if $\Xi = [d]$, i.e., when all coordinates have been updated at least once by $P_{\mathrm{RS}}^{N_d}$. More generally, for all  $x,y \in \R^d$,
		\begin{align*}
			\Pb ( U = V ) \geq  \Pb ( \{U = V\} \cap \{\Xi = \left[ d \right]\} )  
			= \Pb (\Xi = [d]) \cdot  \Pb (U = V \mid \Xi = \left[ d \right]).
		\end{align*}
		By Lemma \ref{couponlem}, $\Pb ( \Xi = \left[ d \right]) \geq 1/2$ when $N_d =  \lceil 2 d \cdot \log d\rceil$. Hence, for this value of $N_d$, we get 
		\begin{align}
			\norm{P_{\mathrm{RS}}^{N_d} (x, \cdot) - P_{\mathrm{RS}}^{N_d} (y, \cdot)}_\tv &\leq \Pb (U \neq V) \nonumber \\
			&= 1 - \Pb (U = V) \nonumber \\
			&\leq 1 - \Pb (\Xi = [d]) \cdot \Pb (U = V \mid \Xi = [d]) \nonumber\\
			&\leq 1 - \frac{1}{2}\cdot  \Pb (U = V \mid \Xi = [d]). \label{rscoupling2}
		\end{align}
		By matching each step of the sequential coupling described above, the conditional probability in \eqref{rscoupling2} can be lower bounded by
		\begin{align}
			\Pb (U = V \mid \Xi = [d]) 	&\geq  \Pb \left( \bigcap_{k=1}^{N_d} \left\{ W_k = W_k'\right\} ~\middle| ~ \Xi = \left[ d \right]\right)  \nonumber \\
			&= \prod_{k=1}^{N_d}  \Pb \left(W_k = W_k' ~\middle| ~ \bigcap_{l=1}^{k-1} \left\{ W_l = W_l'\right\} \cap \left\{\Xi = \left[ d \right]\right\}\right)  \nonumber \\
			&= \prod_{k=1}^{N_d} \left[ 1- \Pb \left(W_k \neq W_k' ~\middle| ~ \bigcap_{l=1}^{k-1} \left\{ W_l = W_l'\right\} \cap \left\{\Xi = \left[ d \right]\right\}\right)\right] \label{rscoupling3}.
		\end{align}
		Consider a fixed a sequence of indices $I_1, \dots, I_{N_d}$ and $k \in \left[ N_d\right]$. Observe that up to the $k$'th step of our sequential coupling, the coordinates of $U^{\left(k\right)}$ that differ from those of $U^{\left(0\right)} = x$ are necessarily equal to some $W_l$ with $l \leq k$. Identically, the coordinates of $V^{\left(k\right)}$ that differ from those of $V^{\left(0\right)} = y$ are necessarily equal to some $W_l'$ with $l \leq k$. Therefore, conditional on the event 
		$
		\bigcap_{l=1}^{k-1} \left\{ W_l = W_l'\right\},
		$
		we have
		$$
		\left \lvert U^{\left(k-1\right)} - V^{\left(k-1\right)} \right \rvert \leq \left \lvert x -y \right \rvert.
		$$
		As $\left(W_k, W_k'\right)$ is a maximal coupling, it follows that 
		\begin{align}
			\Pb &\left( W_k \neq W_k' ~\middle| ~ \bigcap_{l=1}^{k-1}  \left\{ W_l = W_l'\right\} \cap \left\{\Xi = \left[ d \right]\right\}\right) \nonumber \\ 
			&= \E_{U^{\left(k-1\right)}, V^{\left(k-1\right)}, I_k} \left[\norm{ \pi \left( \cdot \mid U^{\left(k-1\right)}_{-{I_k}}\right) -  \pi \left( \cdot \mid V^{\left(k-1\right)}_{-{I_k}}\right)}_\tv ~\middle| ~  \bigcap_{l=1}^{k-1} \left\{ W_l = W_l'\right\} \cap \left\{\Xi = \left[ d \right] \right\} \right] \nonumber \\
			&\leq \E_{U^{\left(k-1\right)}, V^{\left(k-1\right)}, I_k} \left[ M\left \lvert U^{\left(k-1\right)}_{-I_k} - V^{\left(k-1\right)}_{-I_k} \right \rvert^\beta  ~\middle| ~  \bigcap_{l=1}^{k-1} \left\{ W_l = W_l'\right\} \cap \left\{\Xi = \left[ d \right] \right\} \right] \nonumber \\
			&\leq \E_{U^{\left(k-1\right)}, V^{\left(k-1\right)}, I_k} \left[ M \left \lvert U^{\left(k-1\right)} - V^{\left(k-1\right)} \right \rvert^\beta  ~\middle| ~ \bigcap_{l=1}^{k-1} \left\{ W_l = W_l'\right\} \cap \left\{\Xi = \left[ d \right] \right\} \right] \nonumber \\
			&\leq M\left \lvert x - y \right \rvert^\beta, \nonumber
		\end{align}
		where the subscript in the expectation indicates that it is taken with respect to the joint distribution of $U^{(k-1)}, V^{(k-1)}, I_k$.
		Recalling \eqref{rscoupling2} and \eqref{rscoupling3}, we finally obtain
		$$
		\norm{P_{\mathrm{RS}}^{N_d} (x, \cdot) - P_{\mathrm{RS}}^{N_d} (y, \cdot)}_\tv \leq 1 - \frac{1}{2}  \left( 1 - M \left \lvert x-y\right \rvert^\beta\right)^{N_d}.
		$$
	\end{proof}
	
	When a geometric drift condition holds for a certain Lyapunov function, these close coupling conditions are reminiscent of local Doeblin conditions \citep[Assumption 3.4]{hairer2010convergence} restricted to particular level sets of the Lyapunov function; see also \citep[Assumption 2]{hairer2011yet} for an alternative representation in terms of minorization conditions. In this case, one can establish geometric rates of convergence for the Markov process 
	(\citealp[Theorem 3.6]{hairer2010convergence}; \citealp[Theorem 1.2]{hairer2011yet}). Nevertheless, a naive application of Propositions \ref{ssgibbsccc} and \ref{rsgibbsccc} leads to spectral gap bounds that scale exponentially with the dimension.
	It is unclear whether this approach can be improved in the present setting.
	
	\subsection{Mixing time bounds for the Gibbs sampler} \label{sec:mixbounds}
	
	By combining the three-set isoperimetric inequalities from Subsection \ref{sec:3sifi} with the close coupling conditions from Subsection \ref{sec:ccss}, we can bound the conductance of the systematic-scan and random-scan Gibbs kernels. These results apply to target distributions that satisfy an $L^q$-Poincar\'e or $L^q$-log-Sobolev inequality, and whose conditional distributions are TV continuous. Throughout this subsection, we use $\lesssim$ to denote inequalities that hold up to a universal constant factor.
	
	\begin{thm} \label{thm:condss}
		Consider a target distribution $\pi$ that satisfies \po{q} with constant $\cpo{q}$. Suppose that the conditional distributions of $\pi$ are $(M,\beta)$-TV continuous. Then, the conductance of the lazy systematic-scan Gibbs kernel is bounded by
		$$
			\Phi (\widetilde{P}_\mathrm{SS}) \gtrsim \min \left\{1, \frac{\cpo{q}}{2^{2q}(Md)^{q/\beta}} \right\}.
		$$
		If the chain is initialized at a $\Omega$-warm start, then the $\zeta$-mixing time is bounded by
		$$
			\tau (\zeta, \widetilde{P}_{\mathrm{SS}}) \lesssim \max \left\{1, \frac{2^{4q}(Md)^{2q/\beta}}{\cpo{q}^2}\right\} \cdot \log \left( \frac{\sqrt{\Omega}}{\zeta}\right).
		$$
	\end{thm}
	
	\begin{proof}
		Combining Corollary \ref{conductanceProfileLowerBound}, Proposition \ref{3siPI}, and Proposition \ref{ssgibbsccc} yields a lower bound on the conductance of $P_{\mathrm{SS}}$. By \eqref{condlazy}, this directly implies the stated bound on the conductance of $\widetilde{P}_{\mathrm{SS}}$. Since $\widetilde{P}_{\mathrm{SS}}$ is lazy, the bound on the mixing time follows from Corollary \ref{lovsimCor}.
	\end{proof}
	
	\begin{thm} \label{thm:condrsit}
		Consider a target distribution $\pi$ that satisfies \po{q} with constant $\cpo{q}$. Suppose that the conditional distributions of $\pi$ are $(M,\beta)$-TV continuous. Let $N_d =  \lceil 2 d \cdot \log d\rceil$. Then, the conductance of the iterated random-scan Gibbs kernel is bounded by
		$$
			\Phi \left(P^{N_d}_{\mathrm{RS}} \right) \gtrsim \min \left\{ 1, \frac{\cpo{q}}{2^{2q}(MN_d)^{q/\beta}} \right\} .
		$$
		The spectral gap satisfies
		$$
			\lambda_2 \left(P^{N_d}_{\mathrm{RS}} \right) \gtrsim \min \left\{ 1, \frac{\cpo{q}^2}{2^{4q}(MN_d)^{2q/\beta}} \right\} .
		$$
		If the chain is initialized at a $\Omega$-warm start, then the $\zeta$-mixing time is bounded by
		$$
			\tau \left(\zeta, P^{N_d}_{\mathrm{RS}}\right) \lesssim \max \left\{1, \frac{2^{4q}(MN_d)^{2q/\beta}}{\cpo{q}^2} \right\}  \cdot \log \left( \frac{\sqrt{\Omega}}{\zeta}\right).
		$$
	\end{thm}
	
	\begin{proof}
		Combining Corollary \ref{conductanceProfileLowerBound}, Proposition \ref{3siPI}, and Proposition \ref{rsgibbsccc} yields the lower bound on the conductance. The spectral gap is bounded using Theorem \ref{cheegerinequalties}. Since $P_{\mathrm{RS}}$ is reversible and positive semi-definite, the same holds for $P^{N_d}_{\mathrm{RS}}$. The bound on the mixing time then follows from Corollary \ref{lovsimCor}.
	\end{proof}
	
	We now focus on bounding the conductance and spectral gap of the random-scan Gibbs kernel $P_{\mathrm{RS}}$. Recall that $P_{\mathrm{RS}}$ is reversible and positive semi-definite. As noted by \citet[Lemma 5.3]{narayanan2022mixing}, for any reversible Markov kernel $P$ and for all $N \geq 1$, 
	\begin{equation*} \label{condbounditerated}
		\Phi (P) \geq \frac{1}{N} \cdot \Phi \left(P^N\right).
	\end{equation*}
	Let $\sigma (P)$ denote the second largest eigenvalue of $P$, so that $\lambda_2 (P) = 1 -   \sigma (P)$.  If $P$ is positive semi-definite, $\sigma \left(P^N\right) =   \sigma (P)^N$, which implies that
	\begin{align}
		\lambda_2 (P) &= 1 -   \sigma (P) \nonumber \\ 
		&= 1 -   \sigma \left(P^N\right)^{1/N} \nonumber  \\ 
		&= 1 - \left(1 - \lambda_2 \left(P^N\right)\right)^{1/N}\nonumber  \\
		&\geq 1 - \left( 1- \frac{1}{N} \cdot \lambda_2 \left(P^N\right) \right) \nonumber  \\
		&= \frac{1}{N} \cdot \lambda_2 \left(P^N\right) \nonumber,
	\end{align}
	where we have used the inequality $1-t \leq \left(1 - t/N\right)^N$ for $0 \leq t \leq 1$. Regarding mixing times, observe that $\tau (\zeta, P) \leq N \cdot \tau \left(\zeta, P^N\right)$.
	
	\begin{cor} \label{cor:condrs}
		Under the conditions of Theorem \ref{thm:condrsit}, the conductance of the random-scan Gibbs kernel is bounded by
		$$
			\Phi (P_{\mathrm{RS}} ) \gtrsim \frac{1}{d\cdot \log d} \cdot \min \left\{ 1, \frac{\cpo{q}}{2^{2q} (3Md\cdot \log d)^{q/\beta}} \right\}
			.
		$$
		The spectral gap satisfies
		$$
			\lambda_2 (P_{\mathrm{RS}} ) \gtrsim \frac{1}{d\cdot \log d} \cdot \min \left\{ 1, \frac{\cpo{q}^2}{2^{4q} (3Md\cdot \log d)^{2q/\beta}} \right\}
			.
		$$
		The $\zeta$-mixing time is bounded by
		$$
			\tau (\zeta, P_{\mathrm{RS}}) \lesssim d\cdot \log d \cdot \max \left\{ 1, \frac{2^{4q} (3Md\cdot \log d)^{2q/\beta}}{\cpo{q}^2} \right\}
			\cdot \log \left( \frac{\sqrt{\Omega}}{\zeta}\right).
		$$
	\end{cor}
	
	These results demonstrate that the systematic-scan and random-scan Gibbs samplers exhibit polynomial mixing time dependence on the dimension $d$ for a large class of target distributions.
	Note that a single step of the systematic-scan kernel $P_\mathrm{SS}$ updates all $d$ coordinates, whereas the random-scan kernel $P_\mathrm{RS}$ only updates one.
	Even after accounting for this difference, a residual $\log d$ factor remains in the mixing time bounds for random-scan.
	This additional $\log d$ factor
	highlights a modest penalty compared to systematic scan for well-behaved target distributions. This aligns with prior heuristic observations about the efficiency of these sampling schemes \citep{diaconis2013some} and arises naturally from spectral bounds on mixing times \citep{gaitonde2024comparison}.
	
	By Proposition \ref{3siLSI}, all results in this subsection extend to distributions satisfying \ls{q} instead of \po{q}, with the conductance bounds remaining equivalent up to constant factors.
	Consequently, the bounds on the mixing time in total variation distance remain unchanged when replacing a Poincar\'e inequality with a log-Sobolev inequality of the same order.
	
	However, improvements may be possible when using a different norm to define the mixing time.
	One possible approach is  to adapt spectral and conductance profile methods under an $L^q$-log-Sobolev inequality, following  \citet{andrieu2024explicit, chen2020fast}.
	In $\chi^2$-distance, these methods exploit the logarithmic factor in the function $\Psi$ from Proposition \ref{3siLSI} to improve the dependence of the mixing time on the warm start parameter, reducing it from $\log$ to $\log \log$  (see Corollary \ref{lovsimCor}).
	At a high level, this would require sharpening Corollary \ref{conductanceProfileLowerBound} along the lines of \citet[Corollary~16]{andrieu2024explicit} to obtain tighter control of the conductance profile, and then plugging this refinement into the mixing time bounds of \citet[Theorem~8]{andrieu2024explicit}.
	Implementing this approach would require balancing several interacting quantities, such as the order of the Poincar\'e or log-Sobolev inequality, the form of the function $\Psi$ in the associated three-set isoperimetric inequality, and the exponent in the regularity assumption for the conditional distributions (Definition~\ref{def:tvcontinuous}), resulting in less transparent expressions. 
	Moreover, spectral profile methods rely on reversibility and do not improve upon standard spectral gap bounds  for distributions that satisfy only a Poincar\'e inequality \citep[Subsection~3.3]{andrieu2024explicit}. For these reasons, we do not pursue this direction further here.

	\section{Extensions and applications} \label{sec:extensionapp}
	
	\subsection{Implications for Lipschitz and smooth potentials} \label{subsection:potentials}
	
	We now examine the regularity assumption on the conditional distributions of the target~$\pi$ introduced in Definition \ref{def:tvcontinuous}.
	The density of $\pi$ with respect to the Lebesgue measure can be written as	
	$$
	\pi (\mathd x) \propto e^{-U\left(x\right)} \mathd x,
	$$
	where the function $U:\R^d \to \R$ is known as the potential.
	Many previous works assume that the potential function is either Lipschitz continuous \citep{lehec2023langevin} or smooth \citep[see, for example,][]{ascolani2024entropy}. 
	
	\begin{defn} \label{logContDef}
		Let $L > 0$. The potential $U$ is said to be $L$-Lipschitz continuous if
		$$
		\abs{U(x) - U(y)} \leq L \abs{x - y} \qquad \forall x,y \in \R^d.
		$$
		Equivalently, the distribution $\pi$ is said to be $L$-log-Lipschitz continuous.
	\end{defn}
	
	\begin{defn} \label{logSmoothDef}
		Let $L > 0$. The potential $U$ is said to be $L$-smooth if $U$ is continuously differentiable with
		$$
		\abs{\nabla U(x) - \nabla U(y)} \leq L \abs{x - y} \qquad \forall x,y \in \R^d.
		$$
		Equivalently, the distribution $\pi$ is said to be $L$-log-smooth.
	\end{defn}
	
	We show that TV continuity of the conditional distributions of $\pi$ (Definition \ref{def:tvcontinuous}) can be obtained under the above regularity assumptions on $U$. To establish  this, we first demonstrate that log-Lipschitz continuity and log-smoothness are preserved under marginalization, provided that the conditional distributions of $\pi$ satisfy a uniform $L^2$-Poincar\'e inequality. The proof of the following lemma is located in Appendix \ref{sec:appendixproofs}.
	
	\begin{lem} \label{marginalsloglipsmooth}
		The following properties hold:
		\begin{itemize}
			\item Suppose that the distribution $\pi$ is $L$-log-Lipschitz continuous. Then, all marginal distributions of $\pi$ are also $L$-log-Lipschitz continuous.
			\item Suppose that the distribution $\pi$ is $L$-log-smooth and 
				for all $k \in [d]$ and $x\in \R^{d}$, the
				conditional distribution $\pi (\cdot \mid x_{-k})$ satisfies \po{2} with 
				uniform
				constant $\mathtt{c}_2$.
				Then, all $(d-1)$-dimensional marginal distributions are $L'$-log-smooth, where
				$$
				L' = L + \frac{L^2}{\mathtt{c}_2}.
				$$
		\end{itemize}
	\end{lem}
	
	Lemma~\ref{marginalsloglipsmooth} complements the classical result that (strong) log-concavity is preserved under marginalization, as shown in \citet{prekopa1973logarithmic, brascamp1976extensions}; see also \citet{saumard2014log} for a survey.
	The assumption that the conditional distributions of $\pi$ satisfy a uniform $L^2$-Poincar\'e inequality has been widely used in the study of two-scale criteria for functional inequalities \citep{otto2007new, lelievre2009general, grunewald2009two}, in covariance estimates \citep{menz2014brascamp}, and in the analysis of spin systems \citep[see, for example,][]{menz2014approach}.

	As a simple sanity check, the above assumption holds when the target distribution is strongly log-concave.
	More precisely, if $\pi$ is strongly log-concave with constant $m>0$, then each one-dimensional conditional distribution is strongly log-concave with the same constant, and therefore satisfies a uniform $L^2$-Poincaré inequality (see the discussion in Subsection~\ref{subsec:assumptions}).
	Indeed, the potential function of the conditional distribution $\pi (\cdot \mid x_{-k})$ is given by
	$$
	x_k \mapsto U(x_k ,x_{-k}),
	$$
	where $x_{-k}$ is held fixed. Strong log-concavity of $\pi$ with constant $m>0$ is equivalent to the Hessian bound $\nabla^2 U(x) \succeq m I_d$, where $I_d$ denotes the $d$-dimensional identity matrix and $\succeq$ the Loewner order. In particular, this implies that
	$$
	\frac{\partial^2 U }{\partial x^2_k}  (x_k ,x_{-k})  \geq m,
	$$
	which is exactly the condition that $\pi(\cdot \mid x_{-k})$ is strongly log-concave with constant $m$.
	
	\begin{lem} \label{tvconditionalsloglip}
		Suppose that the distribution $\pi$ is $L$-log-Lipschitz continuous. Then, the conditional distributions of $\pi$ are $(\sqrt{L}, 1/2)$-TV continuous.
	\end{lem}
	
	\begin{proof}
		For simplicity,  we redefine the variables as $ x_{-k} = s$, $y_{-k} = t$. The joint distribution can be rewritten as $\pi (\theta, s) \propto \exp({-U(\theta, s)})$. From this, the conditional distribution is expressed as
		$$
		\pi (\theta \mid s ) = \frac{\pi (\theta, s)}{\pi(s)} = \frac{e^{-U(\theta, s)}}{\int e^{-U(\psi, s) } \mathd \psi} = \frac{e^{-U(\theta, s)}}{e^{-V(s)}},
		$$
		where $$V(s) = - \log \left(\int e^{-U(\psi, s) } \mathd \psi\right).$$ By Lemma \ref{marginalsloglipsmooth}, the function $V$ is $L$-Lipschitz continuous.
		We now compute the KL-divergence between $\pi (\cdot \mid s)$ and $\pi (\cdot \mid t)$:
		\begin{align*}
			\mathrm{KL} (\pi (\cdot \mid s),\pi (\cdot \mid t) ) &= \int  \log \left(\frac{\pi (  \theta \mid s )}{\pi (  \theta \mid t )}\right) \pi ( \mathd \theta \mid s ) \\
			&= \int \log\pi (  \theta \mid s )\pi ( \mathd \theta \mid s ) - \int \log \pi (  \theta \mid t ) \pi ( \mathd \theta \mid s ) \\
			&= - \int U(\theta, s) \pi ( \mathd \theta \mid s )  + V(s) +  \int U(\theta, t) \pi ( \mathd \theta \mid s )  - V(t) \\
			&\leq \int  \abs{U(\theta, s) - U(\theta, t) }   \pi ( \mathd \theta \mid s ) +  \abs{V(s) - V(t)} \\
			&\leq 2L \abs{s-t}.
		\end{align*}
		We conclude by applying Pinsker's inequality:
		\begin{align*}
			\norm{ \pi (\cdot \mid s) -  \pi (\cdot \mid t) }_\tv &\leq \sqrt{\frac{1}{2} \cdot \mathrm{KL} (\pi (\cdot \mid s),\pi (\cdot \mid t) )} \\
			&\leq \sqrt{L \abs{s-t}}.\qedhere
		\end{align*} 
	\end{proof}
	
	\begin{lem} \label{tvconditionalslogsmooth}
		Suppose that the distribution $\pi$ is $L$-log-smooth and
		for all $k \in [d]$ and $x\in \R^{d}$, the
		conditional distribution $\pi (\cdot \mid x_{-k})$ satisfies \po{2} with 
		uniform
		constant $\mathtt{c}_2$.
		Then, the conditional distributions of $\pi$ are $(M', 1)$-TV continuous, where
		$$
		M' = \sqrt{\frac{L}{2} + \frac{L^2}{4\mathtt{c}_2}} .
		$$
	\end{lem}
	\begin{proof}
		We proceed similarly to Lemma \ref{tvconditionalsloglip}. Using the same notation, recall that
		\begin{equation} \label{klbound}
			\mathrm{KL} (\pi (\cdot \mid s),\pi (\cdot \mid t) ) 
			= \int ( U(\theta, t) - U(\theta, s) )   \pi ( \mathd \theta \mid s ) +  V(s) - V(t) .
		\end{equation}
		We now bound both differences on the right-hand side of \eqref{klbound}.
		By $L$-smoothness of the function $U$,
		\begin{equation}\label{ubound}
			U(\theta, t) \leq  U(\theta, s) + \nabla_s U (\theta,s)^T (t-s) + \frac{L}{2}\abs{s-t}^2,
		\end{equation}
		where $\nabla_s U(\theta, s)$ represents the gradient of $U(\theta, s)$ with respect to the coordinates of $s$. Moreover, by Lemma \ref{marginalsloglipsmooth}, the function $V$ is $L'$-smooth with $L' = L + L^2/\mathtt{c}_2$. Therefore,
		\begin{equation}\label{vbound}
			V(t) \geq V(s) + \nabla V(s)^T (t-s) - \frac{L'}{2} \abs{s-t}^2. 
		\end{equation}
		Observe that $\nabla V(s) = \int \nabla_s U(\theta, s) \pi(\mathd \theta \mid s)$. Substituting \eqref{ubound} and \eqref{vbound} into \eqref{klbound}, we obtain 
		\begin{align*}
			\mathrm{KL} (\pi (\cdot \mid s),\pi (\cdot \mid t) ) &\leq \int \nabla_s U (\theta,s)^T (t-s) \pi ( \mathd \theta \mid s ) - \nabla V(s)^T (t-s) + \frac{L + L'}{2} \abs{s-t}^2 \\
			&= \left( L + \frac{L^2}{2\mathtt{c}_2}\right)\abs{s-t}^2.
		\end{align*}
		We conclude by applying Pinsker's inequality:
		\begin{align*}
			\norm{ \pi (\cdot \mid s) -  \pi (\cdot \mid t) }_\tv &\leq \sqrt{\frac{1}{2} \cdot \mathrm{KL} (\pi (\cdot \mid s),\pi (\cdot \mid t) )} \\
			&\leq \sqrt{\frac{L}{2} +\frac{L^2}{4\mathtt{c}_2} } \abs{s-t}.
		\end{align*}
	\end{proof}
	
	The class of log-smooth and (strongly) log-concave target distributions is widely studied in the analysis of MCMC mixing times; see \citet{ascolani2024entropy, wadia2024mixing} for results regarding the Gibbs sampler. For comparison, we now establish conductance bounds for the systematic-scan and random-scan Gibbs kernels when the target distribution is log-smooth and satisfies \po{1}. These results follow directly from Theorem \ref{thm:condss} and Corollary \ref{cor:condrs}, combined with Lemma \ref{tvconditionalslogsmooth}.
	
	\begin{cor} \label{cor:condsslog}
		Consider a $L$-log-smooth target distribution $\pi$ that satisfies \po{1} with constant $\cpo{1}$. Suppose that for all $k \in [d]$ and $x\in \R^{d}$, the
		conditional distribution $\pi (\cdot \mid x_{-k})$ satisfies \po{2} with 
		uniform
		constant $\mathtt{c}_2$.
		Then, the conductance of the lazy systematic-scan Gibbs kernel is bounded by
		$$
			\Phi (\widetilde{P}_\mathrm{SS}) \gtrsim \min \left\{ 1, \frac{\sqrt{\mathtt{c}_2}\cpo{1}}{\sqrt{\mathtt{c}_2L + L^2} d} \right\}
			.
		$$
		If the chain is initialized at a $\Omega$-warm start, then the $\zeta$-mixing time is bounded by
		$$
			\tau (\zeta, \widetilde{P}_{\mathrm{SS}}) \lesssim \max \left\{ 1, \frac{(\mathtt{c}_2L + L^2) d^2}{\mathtt{c}_2\cpo{1}^2} \right\}  \cdot \log \left( \frac{\sqrt{\Omega}}{\zeta}\right).
		$$
		
	\end{cor}
	
	However, the bounds in Corollary \ref{thm:condss}  appear to be loose when $\pi$ is a Gaussian distribution, as suggested by comparisons with \citet{roberts1997updating}. 
	
	\begin{cor} \label{cor:condrslog}
		Consider a $L$-log-smooth target distribution $\pi$ that satisfies \po{1} with constant $\cpo{1}$. Suppose that for all $k \in [d]$ and $x\in \R^{d}$, the
		conditional distribution $\pi (\cdot \mid x_{-k})$ satisfies \po{2} with 
		uniform
		constant $\mathtt{c}_2$. 
		Then, the conductance of the random-scan Gibbs kernel is bounded by
		$$
			\Phi (P_{\mathrm{RS}} ) \gtrsim \frac{1}{d\cdot \log d} \cdot \min \left\{ 1, \frac{\sqrt{\mathtt{c}_2}\cpo{1}}{\sqrt{\mathtt{c}_2L + L^2}d \cdot \log d} \right\}
			.
		$$
		The spectral gap satisfies
		$$
			\lambda_2 (P_{\mathrm{RS}} ) \gtrsim \frac{1}{d\cdot \log d} \cdot \min \left\{ 1, \frac{\mathtt{c}_2\cpo{1}^2}{(\mathtt{c}_2L + L^2)(d \cdot \log d)^2} \right\}
			.
		$$
		If the chain is initialized at a $\Omega$-warm start, then the $\zeta$-mixing time is bounded by
		$$
			\tau (\zeta, P_{\mathrm{RS}}) \lesssim d\cdot \log d \cdot \max \left\{ 1, \frac{(\mathtt{c}_2L + L^2) (d \cdot \log d)^2}{\mathtt{c}_2\cpo{1}^2} \right\}
			\cdot \log \left( \frac{\sqrt{\Omega}}{\zeta}\right).
		$$
	\end{cor}
	
	Unfortunately, the spectral gap bound in Corollary \ref{cor:condrs} is suboptimal when $\pi$ is strongly log-concave and $L$-log-smooth. In such cases, \citet[Corollary 3.7]{ascolani2024entropy} established that 
	\begin{equation} \label{sgaprsgibbstight} 
		\lambda_2(P_{\mathrm{RS}}) \gtrsim \frac{\cpo{1}^2}{Ld}. 
	\end{equation}
	Moreover, \citet[Theorem 1]{amit1996convergence} demonstrated that the bound in \eqref{sgaprsgibbstight} is tight when $\pi$ is a Gaussian distribution. 
	The sharp bounds obtained in \citet{ascolani2024entropy} are based on contraction arguments in KL-divergence and critically depend on the assumption of strong log-concavity.
	In particular, it is unclear whether these bounds even extend to bounded perturbations as in \citet{holley1987logarithmic}.
	In contrast, our approach does not exploit such strong structural assumptions, which limits its capacity to achieve comparably sharp bounds. 
	However, this generality allows it to apply to a broader class of target distributions, including those that are not strongly log-concave.
	
	Furthermore, we believe that the lower bound on the spectral gap in Theorem \ref{cheegerinequalties} is not tight for the Gibbs sampler.
	This belief is supported by the case of the random walk on the $d$-dimensional hypercube, for which the conductance and spectral gap are of the same order, making the first bound in Theorem \ref{cheegerinequalties} loose in this setting. 
	The spectral gap for this chain is computed, for example, in \citet[Chapter~3C]{diaconis1988group}, while the conductance is computed in \citet[Example~3.2]{diaconis1991geometric}; see also \citet[Examples~12.16 and~13.11]{levin2017markov} and \citet[Chapter~5]{trevisan2013lecture} for more recent surveys.
	The Gibbs sampler shares a similar local structure to the $d$-dimensional hypercube: each update resamples a single coordinate from its exact conditional, in analogy to the single-bit flips in the hypercube walk.
	
	If the warm start parameter satisfies $\Omega = \Oc(1)$, \citet{wadia2024mixing} proved an upper bound of $\Oc (L^2d^{7.5}(\log d)^2/\cpo{1}^4)$ on the mixing time of the lazy random-scan Gibbs sampler $\widetilde{P}_\mathrm{RS}$. However, a direct comparison of their results with Corollary~\ref{cor:condrslog} is not straightforward. This is  due both to the laziness of the sampler and  the fact that \citet{wadia2024mixing} constrain the Gibbs sampler to a convex subset of $\mathbb{R}^d$, which leads them to analyze mixing via the more restrictive notion of $s$-conductance.

	\subsection{Non-log-concave examples}
	
	Here, we present examples of $d$-dimensional target distributions where our methods yield explicit mixing time bounds for the Gibbs sampler, in contrast to those of \citet{ascolani2024entropy, wadia2024mixing}.
	
	\paragraph{Spherical log-concavity without global log-concavity.}
	
	Consider the distribution
	$$
	\pi_1 (\mathd x) \propto e^{- \abs{\abs{x}-1}} \mathd x.
	$$
	This distribution is log-Lipschitz continuous but not log-concave.
	Writing the potential function of $\pi_1$ as 
	$$
	\abs{\abs{x}-1} = u(\abs{x})
	$$ 
	with $u(t) = \abs{t-1}$, which is convex for $ t >0$, $\pi_1$ is nevertheless spherically log-concave and symmetric in the sense of \citet{bobkov2003spectral, bonnefont2016spectral}. It follows from their results that $\pi_1$ satisfies $\po{2}$.
	
	\paragraph{Perturbed Laplace distribution.} 
	
	Consider the distribution
	$$
	\pi_2 (\mathd x) \propto e^{- \abs{x} + 2 \cos \abs{x}} \mathd x.
	$$
	This distribution is log-Lipschitz continuous and can be written as a bounded perturbation of the Laplace distribution with density proportional to $\exp(-\abs{x})$, which is log-concave and therefore satisfies $\po{1}$ \citep{bobkov1999isoperimetric}. By stability under bounded perturbations, $\pi_2$ also satisfies $\po{1}$; see the discussion in Subsection \ref{subsection:boundingconductance} and, for example, \citet[Proposition~28]{andrieu2024explicit}, where this follows from the equivalence between the $L^1$-Poincar\'e inequality and linear lower bounds on the isoperimetric profile. However, this perturbation prevents $\pi_2$ from being log-concave.
	
	\paragraph{Gaussian mixture.}
	
	Let $m_1, m_2 \in \R^d$ with $m_1 = -m_2 = (1, 1, \dots, 1)$, and consider the two-component mixture of isotropic Gaussian distributions
	$$
	\pi_3 (\mathd x) \propto \left(\frac{1}{2} \cdot e^{- \abs{x-m_1}^2/2} + \frac{1}{2} \cdot e^{- \abs{x-m_2}^2/2}\right)\mathd x.
	$$
	This distribution is log-smooth but not log-concave. Nevertheless, since each Gaussian distribution satisfies \ls{2} \citep{gross1975logarithmic}, it follows from \citet[Theorem 2]{schlichting2019poincare} that $\pi_3$ also satisfies \ls{2}.  Moreover, each conditional distribution of $\pi_3$ can be written as
	$$
	\pi_3( \mathd x_k \mid x_{-k}) \propto \left(w(x_{-k}) \cdot e^{- \abs{x_k-1}^2/2} + (1 - w(x_{-k})) \cdot e^{- \abs{x_k+1}^2/2}\right) \mathd x_k,
	$$
	where $w(x_{-k}) \in [0,1]$. This is a two-component mixture of one-dimensional Gaussian distributions centered at $1$ and $-1$.
	As shown in \citet[Corollary 4.5]{chafai2010fine}; \citet[Corollary 1]{schlichting2019poincare}, the $L^2$-Poincar\'e constant of such mixtures remains uniformly bounded over all values of the mixing weight $w(x_{-k})$. Hence, the conditional distributions of $\pi_3$ satisfy a uniform $L^2$-Poincar\'e inequality.

	\section{Conclusion}
	
	We established explicit conductance bounds for systematic-scan and random-scan Gibbs samplers over a broad class of target distributions, namely those satisfying a Poincaré or log-Sobolev inequality and possessing sufficiently regular conditional distributions (Theorem \ref{thm:condss} and Corollary \ref{cor:condrs}). 	
	These bounds directly provide estimates on the mixing time in total variation distance and remain valid for either log-Lipschitz or log-smooth target distributions.
	In doing so, we extend prior results beyond the log-concave setting \citep{ascolani2024entropy, wadia2024mixing} and present several examples.
	
	Our work also complements recent results showing that the spectral gaps and mixing times of systematic-scan and random-scan Gibbs samplers are closely related \citep{gaitonde2024comparison,chlebicka2023solidarity}.
	While our analysis focuses on coordinate-wise updates, our results naturally extend to blocked Gibbs samplers under corresponding block-smoothness conditions, allowing for joint updates of coordinate subsets.
	Moreover, our work also has implications for Metropolis-within-Gibbs sampling schemes; \citet{ascolani2024scalability} 
	examine how the conductance of generic coordinate-wise schemes compares to that of exact Gibbs sampling, while \citet{qin2023spectral} perform a similar analysis for the spectral gap.
	Their results, combined with our bounds, provide a framework for studying the convergence of broader coordinate-wise sampling methods.
	
	Regarding dimension dependence, it is clear that our bounds on the spectral gap and mixing time are not tight, particularly for Gaussian targets \citep{amit1996convergence, roberts1997updating} and, more generally, for log-smooth and strongly log-concave distributions \citep{ascolani2024entropy}. 
	We attribute this to the generality of our assumptions, which might not fully exploit the underlying structure of these distributions, and to the use of total variation distance in defining the mixing times.
	A natural direction for future research is therefore to refine these conductance bounds for both variants of the Gibbs sampler, particularly in terms of their dependence on the dimension and warmness parameter, as studied in \citet{chen2020fast}.
	One possibility is to study the mixing time of the Gibbs sampler under alternative probability metrics, such as appropriate Wasserstein distances, as in \cite{wang2014convergence}. 
	Beyond the conductance approach, one could also explore recent techniques that may yield sharper spectral gap estimates for the Gibbs sampler, such as quasi-telescoping inequalities (also known as the spectral telescope), as developed in \citet{qin2024spectral}.
	
	\section*{Acknowledgments}
	
	AG was fully funded by the Roth Scholarship provided by the Department of Mathematics, Imperial College London.
	GD was supported during a large part of this work by the Engineering and Physical Sciences Research Council [grant number EP/Y018273/1].

	\bibliographystyle{abbrvnat}
	\bibliography{../../Bib_Cont}

\begin{thebibliography}{107}
\providecommand{\natexlab}[1]{#1}
\providecommand{\url}[1]{\texttt{#1}}
\expandafter\ifx\csname urlstyle\endcsname\relax
  \providecommand{\doi}[1]{doi: #1}\else
  \providecommand{\doi}{doi: \begingroup \urlstyle{rm}\Url}\fi

\bibitem[Altschuler and Chewi(2024{\natexlab{a}})]{altschuler2024faster}
J.~M. Altschuler and S.~Chewi.
\newblock Faster high-accuracy log-concave sampling via algorithmic warm
  starts.
\newblock \emph{Journal of the ACM}, 71\penalty0 (3):\penalty0 1--55,
  2024{\natexlab{a}}.

\bibitem[Altschuler and Chewi(2024{\natexlab{b}})]{altschuler2024shifted}
J.~M. Altschuler and S.~Chewi.
\newblock {Shifted Composition III: Local Error Framework for KL Divergence}.
\newblock \emph{arXiv preprint arXiv:2412.17997}, 2024{\natexlab{b}}.

\bibitem[Altschuler et~al.(2025)Altschuler, Chewi, and
  Zhang]{altschuler2025shifted}
J.~M. Altschuler, S.~Chewi, and M.~S. Zhang.
\newblock {Shifted Composition IV: Underdamped Langevin and Numerical
  Discretizations with Partial Acceleration}.
\newblock \emph{arXiv preprint arXiv:2506.23062}, 2025.

\bibitem[Amit(1996)]{amit1996convergence}
Y.~Amit.
\newblock Convergence properties of the {G}ibbs sampler for perturbations of
  {G}aussians.
\newblock \emph{The Annals of Statistics}, 24\penalty0 (1):\penalty0 122--140,
  1996.

\bibitem[Amit and Grenander(1991)]{amit1991comparing}
Y.~Amit and U.~Grenander.
\newblock Comparing sweep strategies for stochastic relaxation.
\newblock \emph{Journal of Multivariate Analysis}, 37\penalty0 (2):\penalty0
  197--222, 1991.

\bibitem[Anari et~al.(2024)Anari, Chewi, and Vuong]{anari2024fast}
N.~Anari, S.~Chewi, and T.-D. Vuong.
\newblock Fast parallel sampling under isoperimetry.
\newblock In \emph{The Thirty Seventh Annual Conference on Learning Theory},
  pages 161--185. PMLR, 2024.

\bibitem[Andrieu et~al.(2024)Andrieu, Lee, Power, and
  Wang]{andrieu2024explicit}
C.~Andrieu, A.~Lee, S.~Power, and A.~Q. Wang.
\newblock {Explicit convergence bounds for Metropolis Markov chains:
  isoperimetry, spectral gaps and profiles}.
\newblock \emph{The Annals of Applied Probability}, 34\penalty0 (4):\penalty0
  4022--4071, 2024.

\bibitem[Ascolani and Zanella(2024)]{ascolani2024dimension}
F.~Ascolani and G.~Zanella.
\newblock {Dimension-free mixing times of Gibbs samplers for Bayesian
  hierarchical models}.
\newblock \emph{The Annals of Statistics}, 52\penalty0 (3):\penalty0 869--894,
  2024.

\bibitem[Ascolani et~al.(2026{\natexlab{a}})Ascolani, Lavenant, and
  Zanella]{ascolani2024entropy}
F.~Ascolani, H.~Lavenant, and G.~Zanella.
\newblock {Entropy contraction of the Gibbs sampler under log-concavity}.
\newblock \emph{To appear in: The Annals of Probability}, 2026{\natexlab{a}}.

\bibitem[Ascolani et~al.(2026{\natexlab{b}})Ascolani, Roberts, and
  Zanella]{ascolani2024scalability}
F.~Ascolani, G.~O. Roberts, and G.~Zanella.
\newblock {Scalability of Metropolis-within-Gibbs schemes for high-dimensional
  Bayesian models}.
\newblock \emph{To appear in: Journal of the Royal Statistical Society Series
  B: Statistical Methodology}, 2026{\natexlab{b}}.

\bibitem[Bakry and {\'E}mery(1985)]{bakry2006diffusions}
D.~Bakry and M.~{\'E}mery.
\newblock Diffusions hypercontractives.
\newblock \emph{S{\'e}minaire de probabilit{\'e}s de Strasbourg}, 19:\penalty0
  177--206, 1985.

\bibitem[Bakry et~al.(2008)Bakry, Barthe, Cattiaux, and
  Guillin]{bakry2008simple}
D.~Bakry, F.~Barthe, P.~Cattiaux, and A.~Guillin.
\newblock A simple proof of the {P}oincar{\'e} inequality for a large class of
  probability measures.
\newblock \emph{Electronic Communications in Probability}, 13:\penalty0 60--66,
  2008.

\bibitem[Balasubramanian et~al.(2022)Balasubramanian, Chewi, Erdogdu, Salim,
  and Zhang]{balasubramanian2022towards}
K.~Balasubramanian, S.~Chewi, M.~A. Erdogdu, A.~Salim, and S.~Zhang.
\newblock {Towards a Theory of Non-Log-Concave Sampling: First-Order
  Stationarity Guarantees for Langevin Monte Carlo}.
\newblock In \emph{Conference on Learning Theory}, pages 2896--2923. PMLR,
  2022.

\bibitem[Biswas et~al.(2022)Biswas, Bhattacharya, Jacob, and
  Johndrow]{biswas2022coupling}
N.~Biswas, A.~Bhattacharya, P.~E. Jacob, and J.~E. Johndrow.
\newblock Coupling-based convergence assessment of some {G}ibbs samplers for
  high-dimensional {B}ayesian regression with shrinkage priors.
\newblock \emph{Journal of the Royal Statistical Society Series B: Statistical
  Methodology}, 84\penalty0 (3):\penalty0 973--996, 2022.

\bibitem[Bobkov(1999)]{bobkov1999isoperimetric}
S.~G. Bobkov.
\newblock {Isoperimetric and analytic inequalities for log-concave probability
  measures}.
\newblock \emph{The Annals of Probability}, 27\penalty0 (4):\penalty0
  1903--1921, 1999.

\bibitem[Bobkov(2003)]{bobkov2003spectral}
S.~G. Bobkov.
\newblock Spectral gap and concentration for some spherically symmetric
  probability measures.
\newblock In \emph{Geometric Aspects of Functional Analysis: Israel Seminar
  2001-2002}, pages 37--43. Springer, 2003.

\bibitem[Bobkov and Houdr{\'e}(1997)]{bobkov1997some}
S.~G. Bobkov and C.~Houdr{\'e}.
\newblock \emph{{Some Connections between Isoperimetric and Sobolev-Type
  Inequalities}}, volume 129.
\newblock American Mathematical Society, 1997.

\bibitem[Bobkov and Zegarlinski(2005)]{bobkov2005entropy}
S.~G. Bobkov and B.~Zegarlinski.
\newblock \emph{{Entropy Bounds and Isoperimetry}}.
\newblock American Mathematical Society, 2005.

\bibitem[Bonnefont et~al.(2016)Bonnefont, Joulin, and
  Ma]{bonnefont2016spectral}
M.~Bonnefont, A.~Joulin, and Y.~Ma.
\newblock Spectral gap for spherically symmetric log-concave probability
  measures, and beyond.
\newblock \emph{Journal of Functional Analysis}, 270\penalty0 (7):\penalty0
  2456--2482, 2016.

\bibitem[Brascamp and Lieb(1976)]{brascamp1976extensions}
H.~J. Brascamp and E.~H. Lieb.
\newblock {On extensions of the Brunn-Minkowski and Pr{\'e}kopa-Leindler
  theorems, including inequalities for log concave functions, and with an
  application to the diffusion equation}.
\newblock \emph{Journal of Functional Analysis}, 22\penalty0 (4):\penalty0
  366--389, 1976.

\bibitem[Camrud et~al.(2023)Camrud, Durmus, Monmarch{\'e}, and
  Stoltz]{camrud2023second}
E.~Camrud, A.~Durmus, P.~Monmarch{\'e}, and G.~Stoltz.
\newblock {Second order quantitative bounds for unadjusted generalized
  Hamiltonian Monte Carlo}.
\newblock \emph{arXiv preprint arXiv:2306.09513}, 2023.

\bibitem[Chafa{\"\i} and Malrieu(2010)]{chafai2010fine}
D.~Chafa{\"\i} and F.~Malrieu.
\newblock {On fine properties of mixtures with respect to concentration of
  measure and Sobolev type inequalities}.
\newblock \emph{Annales de l’Institut Henri Poincar{\'e} (B) Probabilit{\'e}s
  et Statistiques}, 46\penalty0 (1):\penalty0 72--96, 2010.

\bibitem[Chatterjee(2025)]{chatterjee2023spectral}
S.~Chatterjee.
\newblock Spectral gap of nonreversible {M}arkov chains.
\newblock \emph{The Annals of Applied Probability}, 35\penalty0 (4):\penalty0
  2644--2677, 2025.

\bibitem[Cheeger(1970)]{cheeger1970lower}
J.~Cheeger.
\newblock {A Lower Bound for the Smallest Eigenvalue of the Laplacian}.
\newblock \emph{Problems in Analysis}, 625\penalty0 (195-199):\penalty0 110,
  1970.

\bibitem[Chen et~al.(2021)Chen, Chewi, and Niles-Weed]{chen2021dimension}
H.-B. Chen, S.~Chewi, and J.~Niles-Weed.
\newblock Dimension-free log-{S}obolev inequalities for mixture distributions.
\newblock \emph{Journal of Functional Analysis}, 281\penalty0 (11):\penalty0
  109236, 2021.

\bibitem[Chen and Gatmiry(2023)]{chen2023simple}
Y.~Chen and K.~Gatmiry.
\newblock {A Simple Proof of the Mixing of Metropolis-Adjusted Langevin
  Algorithm under Smoothness and Isoperimetry}.
\newblock \emph{arXiv preprint arXiv:2304.04095}, 2023.

\bibitem[Chen et~al.(2020)Chen, Dwivedi, Wainwright, and Yu]{chen2020fast}
Y.~Chen, R.~Dwivedi, M.~J. Wainwright, and B.~Yu.
\newblock Fast mixing of {M}etropolized {H}amiltonian {M}onte {C}arlo:
  {B}enefits of multi-step gradients.
\newblock \emph{Journal of Machine Learning Research}, 21\penalty0
  (92):\penalty0 1--72, 2020.

\bibitem[Chen et~al.(2022)Chen, Chewi, Salim, and Wibisono]{chen2022improved}
Y.~Chen, S.~Chewi, A.~Salim, and A.~Wibisono.
\newblock Improved analysis for a proximal algorithm for sampling.
\newblock In \emph{Conference on Learning Theory}, pages 2984--3014. PMLR,
  2022.

\bibitem[Chewi et~al.(2021)Chewi, Lu, Ahn, Cheng, Le~Gouic, and
  Rigollet]{chewi2021optimal}
S.~Chewi, C.~Lu, K.~Ahn, X.~Cheng, T.~Le~Gouic, and P.~Rigollet.
\newblock {Optimal dimension dependence of the Metropolis-Adjusted Langevin
  Algorithm}.
\newblock In \emph{Conference on Learning Theory}, pages 1260--1300. PMLR,
  2021.

\bibitem[Chewi et~al.(2025)Chewi, Erdogdu, Li, Shen, and
  Zhang]{chewi2025analysis}
S.~Chewi, M.~A. Erdogdu, M.~Li, R.~Shen, and M.~S. Zhang.
\newblock Analysis of {L}angevin {M}onte {C}arlo from {P}oincar{\'e} to
  {L}og-{S}obolev.
\newblock \emph{Foundations of Computational Mathematics}, 25\penalty0
  (4):\penalty0 1345--1395, 2025.

\bibitem[Chlebicka et~al.(2025)Chlebicka, {\L}atuszy{\'n}ski, and
  Miasojedow]{chlebicka2023solidarity}
I.~Chlebicka, K.~{\L}atuszy{\'n}ski, and B.~Miasojedow.
\newblock {Solidarity of Gibbs samplers: the spectral gap}.
\newblock \emph{The Annals of Applied Probability}, 35\penalty0 (1):\penalty0
  142--157, 2025.

\bibitem[Choi(2020)]{choi2020metropolis}
M.~C. Choi.
\newblock Metropolis-{H}astings reversiblizations of non-reversible {M}arkov
  chains.
\newblock \emph{Stochastic Processes and their Applications}, 130\penalty0
  (2):\penalty0 1041--1073, 2020.

\bibitem[Cousins and Vempala(2014)]{cousins2014cubic}
B.~Cousins and S.~Vempala.
\newblock A cubic algorithm for computing {G}aussian volume.
\newblock In \emph{Proceedings of the twenty-fifth annual ACM-SIAM symposium on
  discrete algorithms}, pages 1215--1228. SIAM, 2014.

\bibitem[Diaconis(1988)]{diaconis1988group}
P.~Diaconis.
\newblock {Group Representations in Probability and Statistics}.
\newblock \emph{Lecture Notes-Monograph Series}, 11:\penalty0 i--192, 1988.

\bibitem[Diaconis(2013)]{diaconis2013some}
P.~Diaconis.
\newblock {Some things we’ve learned (about Markov chain Monte Carlo)}.
\newblock \emph{Bernoulli}, 19\penalty0 (4):\penalty0 1294--1305, 2013.

\bibitem[Diaconis and Stroock(1991)]{diaconis1991geometric}
P.~Diaconis and D.~Stroock.
\newblock Geometric bounds for eigenvalues of {M}arkov chains.
\newblock \emph{The Annals of Applied Probability}, 1\penalty0 (1):\penalty0
  36--61, 1991.

\bibitem[Dolera and Mainini(2020)]{dolera2020uniform}
E.~Dolera and E.~Mainini.
\newblock On uniform continuity of posterior distributions.
\newblock \emph{Statistics \& Probability Letters}, 157:\penalty0 108627, 2020.

\bibitem[Dolera and Mainini(2023)]{dolera2023lipschitz}
E.~Dolera and E.~Mainini.
\newblock Lipschitz continuity of probability kernels in the optimal transport
  framework.
\newblock \emph{Annales de l'Institut Henri Poincar{\'e} (B) Probabilt\'es et
  Statistiques}, 59\penalty0 (4):\penalty0 1778--1812, 2023.

\bibitem[Dwivedi et~al.(2019)Dwivedi, Chen, Wainwright, and Yu]{dwivedi2019log}
R.~Dwivedi, Y.~Chen, M.~J. Wainwright, and B.~Yu.
\newblock Log-concave sampling: {M}etropolis-{H}astings algorithms are fast.
\newblock \emph{Journal of Machine Learning Research}, 20\penalty0
  (183):\penalty0 1--42, 2019.

\bibitem[Dyer et~al.(2008)Dyer, Goldberg, and Jerrum]{dyer2008dobrushin}
M.~Dyer, L.~A. Goldberg, and M.~Jerrum.
\newblock Dobrushin conditions and systematic scan.
\newblock \emph{Combinatorics, Probability and Computing}, 17\penalty0
  (6):\penalty0 761--779, 2008.

\bibitem[Erd{\H{o}}s and R{\'e}nyi(1961)]{erdos1961classical}
P.~Erd{\H{o}}s and A.~R{\'e}nyi.
\newblock On a classical problem of probability theory.
\newblock \emph{A Magyar Tudom{\'a}nyos Akad{\'e}mia Matematikai Kutat{\'o}
  Int{\'e}zet{\'e}nek K{\"o}zlem{\'e}nyei}, 6\penalty0 (1-2):\penalty0
  215--220, 1961.

\bibitem[Fan et~al.(2023)Fan, Yuan, and Chen]{fan2023improved}
J.~Fan, B.~Yuan, and Y.~Chen.
\newblock Improved dimension dependence of a proximal algorithm for sampling.
\newblock In \emph{The Thirty Sixth Annual Conference on Learning Theory},
  pages 1473--1521. PMLR, 2023.

\bibitem[Fernandez(2025)]{fernandez2023ell_0}
V.~M. Fernandez.
\newblock On the {$\ell_0$} {I}soperimetric {C}oefficient for {M}easurable
  {S}ets.
\newblock \emph{Discrete \& Computational Geometry}, pages 1--27, 2025.

\bibitem[Fill(1991)]{fill1991eigenvalue}
J.~A. Fill.
\newblock Eigenvalue bounds on convergence to stationarity for nonreversible
  {M}arkov chains, with an application to the exclusion process.
\newblock \emph{The Annals of Applied Probability}, 1\penalty0 (1):\penalty0
  62--87, 1991.

\bibitem[Gaitonde and Mossel(2026)]{gaitonde2024comparison}
J.~Gaitonde and E.~Mossel.
\newblock {Comparison Theorems for the Mixing Times of Systematic and Random
  Scan Dynamics}.
\newblock In \emph{Proceedings of the 2026 Annual ACM-SIAM Symposium on
  Discrete Algorithms (SODA)}, pages 2575--2587. SIAM, 2026.

\bibitem[Geman and Geman(1984)]{geman1984stochastic}
S.~Geman and D.~Geman.
\newblock {Stochastic relaxation, Gibbs distributions, and the Bayesian
  restoration of images}.
\newblock \emph{IEEE Transactions on Pattern Analysis and Machine
  Intelligence}, \penalty0 (6):\penalty0 721--741, 1984.

\bibitem[Gross(1975)]{gross1975logarithmic}
L.~Gross.
\newblock Logarithmic {S}obolev inequalities.
\newblock \emph{American Journal of Mathematics}, 97\penalty0 (4):\penalty0
  1061--1083, 1975.

\bibitem[Grunewald et~al.(2009)Grunewald, Otto, Villani, and
  Westdickenberg]{grunewald2009two}
N.~Grunewald, F.~Otto, C.~Villani, and M.~G. Westdickenberg.
\newblock A two-scale approach to logarithmic {S}obolev inequalities and the
  hydrodynamic limit.
\newblock \emph{Annales de l'Institut Henri Poincar{\'e} (B) Probabilt\'es et
  Statistiques}, 45\penalty0 (2):\penalty0 302--351, 2009.

\bibitem[Hairer(2010)]{hairer2010convergence}
M.~Hairer.
\newblock Convergence of {M}arkov processes.
\newblock \emph{Lecture notes}, 18\penalty0 (26):\penalty0 11, 2010.

\bibitem[Hairer and Mattingly(2011)]{hairer2011yet}
M.~Hairer and J.~C. Mattingly.
\newblock Yet another look at {H}arris’ ergodic theorem for {M}arkov chains.
\newblock In \emph{Seminar on Stochastic Analysis, Random Fields and
  Applications VI: Centro Stefano Franscini, Ascona, May 2008}, pages 109--117.
  Springer, 2011.

\bibitem[Holley and Stroock(1987)]{holley1987logarithmic}
R.~Holley and D.~Stroock.
\newblock {Logarithmic Sobolev Inequalities and Stochastic Ising Models}.
\newblock \emph{Journal of Statistical Physics}, 46\penalty0 (5-6):\penalty0
  1159--1194, 1987.

\bibitem[Jerrum and Sinclair(1988)]{jerrum1988conductance}
M.~Jerrum and A.~Sinclair.
\newblock Conductance and the rapid mixing property for {M}arkov chains: the
  approximation of permanent resolved.
\newblock In \emph{Proceedings of the Twentieth Annual ACM Symposium on Theory
  of Computing}, pages 235--244, 1988.

\bibitem[Kandasamy and Nagaraj(2024)]{kandasamy2024poisson}
S.~Kandasamy and D.~Nagaraj.
\newblock {The Poisson Midpoint Method for Langevin Dynamics: Provably
  Efficient Discretization for Diffusion Models}.
\newblock \emph{Advances in Neural Information Processing Systems},
  37:\penalty0 65972--66024, 2024.

\bibitem[Kolesnikov(2007)]{kolesnikov2007modified}
A.~V. Kolesnikov.
\newblock Modified log-{S}obolev inequalities and isoperimetry.
\newblock \emph{Rendiconti Lincei: Matematica E Applicazioni}, 18\penalty0
  (2):\penalty0 179--208, 2007.

\bibitem[Kontoyiannis and Meyn(2012)]{kontoyiannis2012geometric}
I.~Kontoyiannis and S.~P. Meyn.
\newblock {Geometric ergodicity and the spectral gap of non-reversible Markov
  chains}.
\newblock \emph{Probability Theory and Related Fields}, 154\penalty0
  (1):\penalty0 327--339, 2012.

\bibitem[Laddha and Vempala(2023)]{laddha2023convergence}
A.~Laddha and S.~S. Vempala.
\newblock {Convergence of Gibbs Sampling: Coordinate Hit-and-Run Mixes Fast}.
\newblock \emph{Discrete \& Computational Geometry}, 70\penalty0 (2):\penalty0
  406--425, 2023.

\bibitem[Lawler and Sokal(1988)]{lawler1988bounds}
G.~F. Lawler and A.~D. Sokal.
\newblock Bounds on the ${L}^2$-spectrum for {M}arkov chains and {M}arkov
  processes: a generalization of {C}heeger’s inequality.
\newblock \emph{Transactions of the American Mathematical Society},
  309\penalty0 (2):\penalty0 557--580, 1988.

\bibitem[Ledoux(1999)]{ledoux2006concentration}
M.~Ledoux.
\newblock Concentration of measure and logarithmic {S}obolev inequalities.
\newblock \emph{S{\'e}minaire de probabilit{\'e}s de Strasbourg}, 33:\penalty0
  120--216, 1999.

\bibitem[Ledoux(2001)]{ledoux2001concentration}
M.~Ledoux.
\newblock \emph{The Concentration of Measure Phenomenon}.
\newblock Number~89. American Mathematical Society, 2001.

\bibitem[Lee and Zhang(2024)]{lee2024fast}
H.~Lee and K.~Zhang.
\newblock Fast mixing of data augmentation algorithms: {B}ayesian probit,
  logit, and lasso regression.
\newblock \emph{arXiv preprint arXiv:2412.07999}, 2024.

\bibitem[Lehec(2023)]{lehec2023langevin}
J.~Lehec.
\newblock {The Langevin Monte Carlo algorithm in the non-smooth log-concave
  case}.
\newblock \emph{The Annals of Applied Probability}, 33\penalty0 (6A):\penalty0
  4858--4874, 2023.

\bibitem[Lehec(2025)]{lehec2025convergence}
J.~Lehec.
\newblock Convergence in total variation for the kinetic {L}angevin algorithm.
\newblock \emph{Mathematical Statistics and Learning}, 2025.

\bibitem[Leli{\`e}vre(2009)]{lelievre2009general}
T.~Leli{\`e}vre.
\newblock A general two-scale criteria for logarithmic {S}obolev inequalities.
\newblock \emph{Journal of Functional Analysis}, 256\penalty0 (7):\penalty0
  2211--2221, 2009.

\bibitem[Levin and Peres(2017)]{levin2017markov}
D.~A. Levin and Y.~Peres.
\newblock \emph{Markov Chains and Mixing Times}, volume 107.
\newblock American Mathematical Society, 2nd edition, 2017.

\bibitem[Liang and Chen(2022)]{liang2022proximal}
J.~Liang and Y.~Chen.
\newblock {A Proximal Algorithm for Sampling}.
\newblock \emph{arXiv preprint arXiv:2202.13975}, 2022.

\bibitem[Lindvall(2002)]{lindvall2002lectures}
T.~Lindvall.
\newblock \emph{{Lectures on the Coupling Method}}.
\newblock Courier Corporation, 2002.

\bibitem[Lov{\'a}sz(1999)]{lovasz1999hit}
L.~Lov{\'a}sz.
\newblock Hit-and-run mixes fast.
\newblock \emph{Mathematical Programming}, 86:\penalty0 443--461, 1999.

\bibitem[Lov{\'a}sz and Simonovits(1993)]{lovasz1993random}
L.~Lov{\'a}sz and M.~Simonovits.
\newblock Random walks in a convex body and an improved volume algorithm.
\newblock \emph{Random Structures \& Algorithms}, 4\penalty0 (4):\penalty0
  359--412, 1993.

\bibitem[Lov{\'a}sz and Vempala(2004)]{lovasz2004hit}
L.~Lov{\'a}sz and S.~Vempala.
\newblock Hit-and-run from a corner.
\newblock In \emph{Proceedings of the Thirty-Sixth Annual ACM Symposium on
  Theory of Computing}, pages 310--314, 2004.

\bibitem[Ma et~al.(2021)Ma, Chatterji, Cheng, Flammarion, Bartlett, and
  Jordan]{ma2021there}
Y.-A. Ma, N.~S. Chatterji, X.~Cheng, N.~Flammarion, P.~L. Bartlett, and M.~I.
  Jordan.
\newblock {Is there an analog of Nesterov acceleration for gradient-based
  MCMC?}
\newblock \emph{Bernoulli}, 27\penalty0 (3):\penalty0 1942--1992, 2021.

\bibitem[Mangoubi and Vishnoi(2019)]{mangoubi2019nonconvex}
O.~Mangoubi and N.~K. Vishnoi.
\newblock {Nonconvex sampling with the Metropolis-adjusted Langevin algorithm}.
\newblock In \emph{Conference on Learning Theory}, pages 2259--2293. PMLR,
  2019.

\bibitem[Marton(2004)]{marton2004measure}
K.~Marton.
\newblock Measure concentration for {E}uclidean distance in the case of
  dependent random variables.
\newblock \emph{The Annals of Probability}, 32\penalty0 (1A):\penalty0
  2526--2544, 2004.

\bibitem[Menz(2014{\natexlab{a}})]{menz2014approach}
G.~Menz.
\newblock The approach of {O}tto-{R}eznikoff revisited.
\newblock \emph{Electronic Journal of Probability}, 19\penalty0 (107):\penalty0
  1--27, 2014{\natexlab{a}}.

\bibitem[Menz(2014{\natexlab{b}})]{menz2014brascamp}
G.~Menz.
\newblock A {B}rascamp-{L}ieb type covariance estimate.
\newblock \emph{Electronic Journal of Probability}, 19\penalty0 (78):\penalty0
  1--15, 2014{\natexlab{b}}.

\bibitem[Milman(2009{\natexlab{a}})]{milman2009role}
E.~Milman.
\newblock {On the Role of Convexity in Isoperimetry, Spectral Gap and
  Concentration}.
\newblock \emph{Inventiones Mathematicae}, 177\penalty0 (1):\penalty0 1--43,
  2009{\natexlab{a}}.

\bibitem[Milman(2009{\natexlab{b}})]{milman2009role2}
E.~Milman.
\newblock {On the Role of Convexity in Functional and Isoperimetric
  Inequalities}.
\newblock \emph{Proceedings of the London Mathematical Society}, 99\penalty0
  (1):\penalty0 32--66, 2009{\natexlab{b}}.

\bibitem[Mitra and Wibisono(2025)]{mitra2025fast}
S.~Mitra and A.~Wibisono.
\newblock Fast {C}onvergence of {$\Phi$}-{D}ivergence {A}long the {U}nadjusted
  {L}angevin {A}lgorithm and {P}roximal {S}ampler.
\newblock In \emph{Algorithmic Learning Theory}, pages 846--869. PMLR, 2025.

\bibitem[Mou et~al.(2022{\natexlab{a}})Mou, Flammarion, Wainwright, and
  Bartlett]{mou2022efficient}
W.~Mou, N.~Flammarion, M.~J. Wainwright, and P.~L. Bartlett.
\newblock {An Efficient Sampling Algorithm for Non-Smooth Composite
  Potentials}.
\newblock \emph{Journal of Machine Learning Research}, 23\penalty0
  (233):\penalty0 1--50, 2022{\natexlab{a}}.

\bibitem[Mou et~al.(2022{\natexlab{b}})Mou, Flammarion, Wainwright, and
  Bartlett]{mou2022improved}
W.~Mou, N.~Flammarion, M.~J. Wainwright, and P.~L. Bartlett.
\newblock {Improved bounds for discretization of Langevin diffusions:
  Near-optimal rates without convexity}.
\newblock \emph{Bernoulli}, 28\penalty0 (3):\penalty0 1577--1601,
  2022{\natexlab{b}}.

\bibitem[Narayanan and Srivastava(2022)]{narayanan2022mixing}
H.~Narayanan and P.~Srivastava.
\newblock On the mixing time of coordinate hit-and-run.
\newblock \emph{Combinatorics, Probability and Computing}, 31\penalty0
  (2):\penalty0 320--332, 2022.

\bibitem[Narayanan et~al.(2024)Narayanan, Rajaraman, and
  Srivastava]{narayanan2024sampling}
H.~Narayanan, A.~Rajaraman, and P.~Srivastava.
\newblock Sampling from convex sets with a cold start using multiscale
  decompositions.
\newblock \emph{Probability Theory and Related Fields}, pages 1--64, 2024.

\bibitem[Otto and Reznikoff(2007)]{otto2007new}
F.~Otto and M.~G. Reznikoff.
\newblock A new criterion for the logarithmic {S}obolev inequality and two
  applications.
\newblock \emph{Journal of Functional Analysis}, 243\penalty0 (1):\penalty0
  121--157, 2007.

\bibitem[Paulin(2015)]{paulin2015concentration}
D.~Paulin.
\newblock Concentration inequalities for {M}arkov chains by {M}arton couplings
  and spectral methods.
\newblock \emph{Electronic Journal of Probability}, 20:\penalty0 79, 2015.

\bibitem[Pr{\'e}kopa(1973)]{prekopa1973logarithmic}
A.~Pr{\'e}kopa.
\newblock On logarithmic concave measures and functions.
\newblock \emph{Acta Scientiarum Mathematicarum}, 34:\penalty0 335--343, 1973.

\bibitem[Qin and Hobert(2019)]{qin2019convergence}
Q.~Qin and J.~P. Hobert.
\newblock {Convergence complexity analysis of Albert and Chib’s algorithm for
  Bayesian probit regression}.
\newblock \emph{The Annals of Statistics}, 47\penalty0 (4):\penalty0
  2320--2347, 2019.

\bibitem[Qin and Hobert(2022)]{qin2022wasserstein}
Q.~Qin and J.~P. Hobert.
\newblock {Wasserstein-based methods for convergence complexity analysis of
  MCMC with applications}.
\newblock \emph{The Annals of Applied Probability}, 32\penalty0 (1):\penalty0
  124--166, 2022.

\bibitem[Qin and Wang(2024)]{qin2024spectral}
Q.~Qin and G.~Wang.
\newblock {Spectral telescope: Convergence rate bounds for random-scan Gibbs
  samplers based on a hierarchical structure}.
\newblock \emph{The Annals of Applied Probability}, 34\penalty0 (1B):\penalty0
  1319--1349, 2024.

\bibitem[Qin et~al.(2025)Qin, Ju, and Wang]{qin2023spectral}
Q.~Qin, N.~Ju, and G.~Wang.
\newblock {Spectral gap bounds for reversible hybrid Gibbs chains}.
\newblock \emph{The Annals of Statistics}, 53\penalty0 (4):\penalty0
  1613--1638, 2025.

\bibitem[Roberts and Rosenthal(2013)]{roberts2013note}
G.~O. Roberts and J.~S. Rosenthal.
\newblock A note on formal constructions of sequential conditional couplings.
\newblock \emph{Statistics \& Probability Letters}, 83\penalty0 (9):\penalty0
  2073--2076, 2013.

\bibitem[Roberts and Sahu(1997)]{roberts1997updating}
G.~O. Roberts and S.~K. Sahu.
\newblock Updating schemes, correlation structure, blocking and
  parameterization for the {G}ibbs sampler.
\newblock \emph{Journal of the Royal Statistical Society Series B: Statistical
  Methodology}, 59\penalty0 (2):\penalty0 291--317, 1997.

\bibitem[Roberts and Smith(1994)]{roberts1994simple}
G.~O. Roberts and A.~F. Smith.
\newblock {Simple conditions for the convergence of the Gibbs sampler and
  Metropolis-Hastings algorithms}.
\newblock \emph{Stochastic Processes and their Applications}, 49\penalty0
  (2):\penalty0 207--216, 1994.

\bibitem[Rothaus(1985)]{rothaus1985analytic}
O.~Rothaus.
\newblock {Analytic Inequalities, Isoperimetric Inequalities and Logarithmic
  Sobolev Inequalities}.
\newblock \emph{Journal of Functional Analysis}, 64\penalty0 (2):\penalty0
  296--313, 1985.

\bibitem[Saumard and Wellner(2014)]{saumard2014log}
A.~Saumard and J.~A. Wellner.
\newblock Log-concavity and strong log-concavity: a review.
\newblock \emph{Statistics Surveys}, 8:\penalty0 45, 2014.

\bibitem[Schlichting(2019)]{schlichting2019poincare}
A.~Schlichting.
\newblock {Poincar{\'e} and log-Sobolev inequalities for mixtures}.
\newblock \emph{Entropy}, 21\penalty0 (1):\penalty0 89, 2019.

\bibitem[Trevisan(2013)]{trevisan2013lecture}
L.~Trevisan.
\newblock {Lecture Notes on Expansion, Sparsest Cut, and Spectral Graph
  Theory}.
\newblock \emph{URL: \url{https://lucatrevisan.github.io/books/expanders.pdf}},
  2013.

\bibitem[Vempala(2005)]{vempala2005geometric}
S.~Vempala.
\newblock Geometric random walks: a survey.
\newblock \emph{Combinatorial and Computational Geometry}, 52\penalty0
  (573-612):\penalty0 2, 2005.

\bibitem[Vempala and Wibisono(2019)]{vempala2019rapid}
S.~Vempala and A.~Wibisono.
\newblock {Rapid Convergence of the Unadjusted Langevin Algorithm: Isoperimetry
  Suffices}.
\newblock \emph{Advances in Neural Information Processing Systems}, 32, 2019.

\bibitem[Wadia(2024)]{wadia2024mixing}
N.~S. Wadia.
\newblock A mixing time bound for {G}ibbs sampling from log-smooth log-concave
  distributions.
\newblock \emph{arXiv preprint arXiv:2412.17899}, 2024.

\bibitem[Wang(2017)]{wang2017convergence}
N.-Y. Wang.
\newblock Convergence rates of the random scan {G}ibbs sampler under the
  {D}obrushin’s uniqueness condition.
\newblock \emph{Electronic Communications in Probability}, 22\penalty0
  (56):\penalty0 1--7, 2017.

\bibitem[Wang(2019)]{wang2019convergence}
N.-Y. Wang.
\newblock Convergence rates of symmetric scan {G}ibbs sampler.
\newblock \emph{Frontiers of Mathematics in China}, 14:\penalty0 941--955,
  2019.

\bibitem[Wang and Wu(2014)]{wang2014convergence}
N.-Y. Wang and L.~Wu.
\newblock Convergence rate and concentration inequalities for {Gibbs} sampling
  in high dimension.
\newblock \emph{Bernoulli}, pages 1698--1716, 2014.

\bibitem[Wibisono(2019)]{wibisono2019proximal}
A.~Wibisono.
\newblock {Proximal Langevin Algorithm: Rapid Convergence under Isoperimetry}.
\newblock \emph{arXiv preprint arXiv:1911.01469}, 2019.

\bibitem[Wu et~al.(2022)Wu, Schmidler, and Chen]{wu2022minimax}
K.~Wu, S.~Schmidler, and Y.~Chen.
\newblock {Minimax mixing time of the Metropolis-adjusted Langevin algorithm
  for log-concave sampling}.
\newblock \emph{Journal of Machine Learning Research}, 23\penalty0
  (270):\penalty0 1--63, 2022.

\bibitem[Wu(2006)]{wu2006poincare}
L.~Wu.
\newblock Poincar{\'e} and transportation inequalities for {G}ibbs measures
  under the {D}obrushin uniqueness condition.
\newblock \emph{The Annals of Probability}, 34:\penalty0 1960--1989, 2006.

\bibitem[Yang and Rosenthal(2023)]{yang2023complexity}
J.~Yang and J.~S. Rosenthal.
\newblock {Complexity results for MCMC derived from quantitative bounds}.
\newblock \emph{The Annals of Applied Probability}, 33\penalty0 (2):\penalty0
  1459--1500, 2023.

\bibitem[Zhang et~al.(2023)Zhang, Chewi, Li, Balasubramanian, and
  Erdogdu]{zhang2023improved}
S.~Zhang, S.~Chewi, M.~Li, K.~Balasubramanian, and M.~A. Erdogdu.
\newblock {Improved Discretization Analysis for Underdamped Langevin Monte
  Carlo}.
\newblock In \emph{The Thirty Sixth Annual Conference on Learning Theory},
  pages 36--71. PMLR, 2023.

\bibitem[Zou et~al.(2021)Zou, Xu, and Gu]{zou2021faster}
D.~Zou, P.~Xu, and Q.~Gu.
\newblock {Faster Convergence of Stochastic Gradient Langevin Dynamics for
  Non-Log-Concave Sampling}.
\newblock In \emph{Uncertainty in Artificial Intelligence}, pages 1152--1162.
  PMLR, 2021.

\end{thebibliography}
	
	\appendix
	
	\section{Appendix} \label{sec:appendixproofs}
	
	\begin{proof}[Proof of Proposition \ref{LSIhierarchy}]
		Define the function $N_q (x) = \abs{x}^q \cdot \log ( 1 + \abs{x}^q)$ for all $x\in \R^d$. Given a function $f : \R^d \to \R$, denote its Orlicz norm associated to $N_q$ by
		$$
		\norm{f}_{N_q} = \inf \left\{ \lambda > 0 : \int N_q \left(\frac{f(x)}{\lambda}\right) \pi (\mathd x) \leq 1\right\}.
		$$
		Let $\norm{f}_q$ also denote the $L^q_\pi$-norm of $f$, as defined in Subsection \ref{sec:3sifi}. We now  introduce two functional inequalities in the Orlicz space. Let  $\mathtt{k}_q > 0$ and $\tilde{\mathtt{k}}_q > 0$ denote the largest constants such that, for all locally Lipschitz functions $f : \R^d \to \R$,
		\begin{align}
			\mathtt{k}_q \norm{f - \E [f]}_{N_q}^q &\leq \E \left[ \abs{\nabla f  }^q\right], \label{lsiorliczmean} \\
			\tilde{\mathtt{k}}_q \norm{f - m [f]}_{N_q}^q &\leq \E \left[ \abs{\nabla f  }^q\right], \label{lsiorliczmedian}
		\end{align}
		where $m [f]$ denotes a median of $f$. \citet[Proposition 3.1]{bobkov2005entropy} showed that \ls{q} is equivalent to the functional inequality \eqref{lsiorliczmean}, with constants satisfying \hbox{$\mathtt{k}_q/16\leq \cls{q} \leq 7 \mathtt{k}_q$}.
		Moreover, by \citet[Lemma 2.1]{milman2009role},  $\tilde{\mathtt{k}}_q/2 \leq \mathtt{k}_q \leq 3 \tilde{\mathtt{k}}_q$.
		The functional inequality \eqref{lsiorliczmedian} is therefore equivalent to \ls{q}, with constants satisfying
		\begin{equation} \label{lsilsimedianconstants}
			\frac{\tilde{\mathtt{k}}_q}{32} \leq \cls{q} \leq 21\tilde{\mathtt{k}}_q.
		\end{equation}
		Now, take any locally Lipschitz function $f$ with $m[f] = 0$ and define $g = \mathrm{sgn}(f)\cdot \abs{f}^{r/q}$, where $\mathrm{sgn} (\cdot)$ denotes the sign function. In particular, we  have $m[g] = 0$ as well. Applying \eqref{lsiorliczmedian} to $g$ yields
		\begin{equation} \label{lsiG}
			{	\tilde{\mathtt{k}}_q}^{1/q}  \norm{g}_{N_q} \leq \norm{\abs{\nabla g }}_q.
		\end{equation}
		Since the solution to the equation $1/u + 1/r = 1/q$ is $u = qr/(r-q)$, an application of H\"older's inequality gives
		\begin{align*}
			\norm{\abs{\nabla g }}_q &\leq \frac{r}{q}  \norm{\abs{f}^{r/q-1} \cdot \abs{\nabla f}}_q \\
			&= \frac{r}{q} \norm{\abs{f}^{r/q-1}}_{qr/(r-q)} \cdot \norm{\abs{\nabla f}}_r \\
			&= \frac{r}{q} \norm{f}_r^{r/q -1} \cdot \norm{\abs{\nabla f}}_r \\
			&\leq \frac{r}{q} \left(\frac{5}{4}\right)^{1/q - 1/r}   \norm{f}_{N_r}^{r/q-1} \cdot \norm{\abs{\nabla f}}_r.
		\end{align*}
		To justify the transition from the second to the third line, observe that
		$$
		\left(\frac{r}{q} -1\right) \frac{qr}{r-q} = \frac{r^2-qr}{r-q} = r
		$$
		and
		$$
		\frac{r-q}{qr} = \frac{1}{r} \left(\frac{r}{q} - 1\right).
		$$
		The last inequality follows from \citet[Lemma 3.2]{bobkov2005entropy}.
		Furthermore, for all $\lambda > 0$, 
		\begin{align*}
			N_q \left(\frac{g}{\lambda}\right) &=  \frac{\abs{f}^r}{\lambda^q} \cdot \log \left(1 + \frac{\abs{f}^r}{\lambda^q}\right) \\
			&= \abs{\frac{f}{\lambda^{q/r}}}^r \cdot \log \left(1 + \abs{ \frac{f}{\lambda^{q/r}}}^r\right) \\
			&= N_r \left(\frac{f}{\lambda^{q/r}}\right).
		\end{align*}
		It follows that
		\begin{align*}
			\norm{g}_{N_q} &= \inf \left\{ \lambda > 0 : \int N_q \left( \frac{ g(x)}{\lambda} \right) \pi(\mathd x) \leq 1\right\} \\
			&= \inf \left\{ \lambda > 0 : \int N_r \left( \frac{f(x)}{\lambda^{q/r}}  \right) \pi(\mathd x) \leq 1\right\} \\
			&= \inf \left\{ \lambda^{q/r} > 0 : \int N_r \left( \frac{f(x)}{\lambda^{q/r}} \right) \pi(\mathd x) \leq 1\right\}^{r/q} \\
			&= \inf \left\{ \lambda' > 0 : \int N_r \left( \frac{f(x)}{\lambda'} \right) \pi(\mathd x) \leq 1\right\}^{r/q} \\
			&= \norm{f}_{N_r}^{r/q}.
		\end{align*}
		Inserting these identities into \eqref{lsiG} gives
		$$
		{	\tilde{\mathtt{k}}_q}^{1/q} \norm{f}_{N_r}^{r/q} \leq \left(\frac{5}{4}\right)^{1/q - 1/r}  \frac{r}{q}    \norm{f}_{N_r}^{r/q-1} \cdot \norm{\abs{\nabla f}}_r.
		$$
		Rearranging terms and raising both sides to the power $r$, we obtain
		$$
		\left(\frac{4}{5}\right)^{r/q-1} \left(\frac{q}{r}\right)^r  {	\tilde{\mathtt{k}}_q}^{r/q} \norm{f}_{N_r}^r \leq  \E \left[ \abs{\nabla f  }^r\right].
		$$
		Hence,
		\begin{equation*}
			\left(\frac{4}{5}\right)^{r/q-1} \left(\frac{q}{r}\right)^r  {	\tilde{\mathtt{k}}_q}^{r/q} \leq \tilde{\mathtt{k}}_r.
		\end{equation*}
		Finally, applying \eqref{lsilsimedianconstants} twice yields
		$$
		\frac{5}{128} \left(\frac{4}{105}\right)^{r/q}
		\left(\frac{q}{r}\right)^r \cls{q}^{r/q} \leq \cls{r}.
		$$
	\end{proof}
	
	\begin{proof}[Proof of Proposition \ref{conductanceTHM}]
		Let $S$ $\subseteq \R^d$ with $0 <\pi ( S) \leq 1/2$. Note that in this case, \mbox{$\pi ( S^c  ) \geq \pi( S )$}. Define the sets
		\begin{align*}
			S_1 &= \left\{ x \in S  : P (x, S^c) \leq \frac{\varepsilon}{2}\right\}, \\
			S_2 &= \left\{ x \in S^c  : P (x, S) \leq \frac{\varepsilon}{2}\right\}, \\
			S_3 &= \left( S_1 \sqcup S_2 \right)^c.
		\end{align*}
		Here, $S_1$ and $S_2$ can be viewed as ``bad'' sets: in $S_1$ (respectively $S_2$), the probability to escape $S $ (respectively $S^c$) is small. Note that $S_1 \sqcup S_2 \sqcup S_3 = \R^d$, forming a partition of space. We now separate two different cases.
		
		In the first case, suppose that $\pi ( S_1 ) \leq \pi ( S)/2$ or $\pi ( S_2 ) \leq \pi ( S^c)/2$, i.e., at least one of the ``bad'' sets is small. If $\pi ( S_1 ) \leq \pi ( S)/2$, then
		\begin{align*}
			\int_S  P   (x, S^c ) \pi(\mathd x ) &\geq \int_{  S  \setminus S_1}  P   ( x, S^c ) \pi   (\mathd x ) \\
			& \geq \frac{\varepsilon}{2}\cdot \pi   (   S   \setminus S_1  ) \\
			& = \frac{\varepsilon}{2}   \left[ \pi   (S  ) - \pi   ( S_1  )  \right] \\
			& \geq \frac{\varepsilon}{4}\cdot \pi   (S  ).
		\end{align*}
		If $\pi   ( S_2  ) \leq \pi   ( S^c   )/2$, we proceed analogously: 
		\begin{align*}
			\int_S P   (x, S^c ) \pi   (\mathd x )  &= \int_{S^c}  P   (x, S )  \pi   (\mathd x ) \\
			&\geq \int_{  S^c   \setminus S_2} P   ( x, S^c )  \pi   (\mathd x )  \\
			& \geq \frac{\varepsilon}{2}\cdot \pi   (   S^c   \setminus S_2  ) \\
			& = \frac{\varepsilon}{2}   \left[ \pi   (S^c  )- \pi   ( S_2  )  \right] \\
			& \geq \frac{\varepsilon}{4} \cdot\pi   ( S^c    ) \\
			& \geq \frac{\varepsilon}{4}\cdot \pi   ( S   ).
		\end{align*}
		Importantly, reversibility is not required to establish the first equality, as noted by \citet[Lemma 2]{dwivedi2019log}.
		
		In the second case, suppose that $\pi   ( S_1  ) > \pi   ( S  )/2$ and $\pi   ( S_2  ) > \pi   ( S^c  )/2$, i.e., both of the ``bad'' sets are large. Take $x \in S_1$, $y \in S_2$. We have
		$$
		\norm{P (x, \cdot ) - P (y, \cdot)}_\tv \geq P   ( x, S  ) - P   ( y, S  ) = 1 - P   ( x, S^c  ) - P   ( y, S  ) \geq 1 - \varepsilon.
		$$
		By the contrapositive of the $( \delta, \varepsilon  )$-close coupling condition, this implies that $\abs{x-y} >\delta$, and consequently, $\mathcal{D}  ( S_1, S_2 ) \geq\delta$. This leads to
		\begin{align*}
			\int_S P  ( x, S^c ) \pi   (\mathd x )  &= \frac{1}{2}   \left[ \int_S  P  ( x, S^c ) \pi   (\mathd x )  + 	\int_{S^c} P  ( x, S ) \pi   (\mathd x )    \right] \\
			&\geq \frac{1}{2}   \left[ \int_{  S    \setminus S_1} P  ( x, S^c ) \pi   (\mathd x )  + \int_{  S^c    \setminus S_2} P  ( x, S ) \pi   (\mathd x )    \right]\\
			&\geq \frac{\varepsilon}{4}   \left[ \pi   (   S    \setminus S_1  ) + \pi   (   S^c    \setminus S_2  ) \right] \\
			&= \frac{\varepsilon}{4} \cdot \pi   ( S_3 ) \\
			&\geq \frac{\varepsilon}{4} \cdot \Upsilon  \left(\mathcal{D}  (S_1, S_2 ) \right) \cdot \Psi   \left(\min   \left\{ \pi   ( S_1 ), \pi   ( S_2 )  \right\} \right) \\
			&\geq \frac{\varepsilon}{4} \cdot \Upsilon  (\delta ) \cdot \Psi  \left( \frac{\pi (S  )}{2}  \right).
		\end{align*}
		The statement follows by taking minimums of both cases.
	\end{proof}
	
	\begin{proof}[Proof of Lemma \ref{3siLSIlem}]
		Since $\rho(0) = 0$ and $\log(1/x) \geq 0$ for $x \in (0,1]$, the function $\rho$ is positive.
		The first and second derivatives of $\rho$ are given by
		\begin{align*}
			\rho'(x) = \log \left(\frac{1}{x}\right) - 1, \qquad
			\rho''(x) = - \frac{1}{x}.
		\end{align*}
		Since $\rho''$ is negative on its domain, the function $\rho$ is concave. Moreover, $\rho'(x)=0$ at $x=e^{-1}$, and hence $\rho$ attains its maximum at $x=e^{-1}$. This proves the first part of the lemma.
		For the second part, by concavity of $\rho$, a Taylor expansion at $x+h$ yields
		$$
		\rho(x) \leq \rho(x+h) - h\rho'(x+h).
		$$
		Since $\rho'(x+h) \geq -1$, it follows that $\rho(x) \leq \rho(x+h) + h$, and therefore
		$$
		\rho(x) -h \leq \rho(x+h).
		$$
	\end{proof}
	
	\begin{proof}[Proof of Lemma \ref{couponlem}]
		For any $k\in [d]$, the probability that $k$ is not drawn is
		$$
		\Pb ( k \notin \Xi) = \left(1- \frac{1}{d}\right)^N \leq e^{-N/d}.
		$$
		By a union bound,
		\begin{align*}
			\Pb (\Xi = [d]) &= 1 - \Pb (\Xi \neq [d])\\
			&= 1 - \Pb \left(\bigcup_{k=1}^d \{ k \notin \Xi\}\right) \\
			&\geq 1 - \sum_{k=1}^d \Pb ( k \notin \Xi) \\
			&\geq 1 - d e^{-N/d}.
		\end{align*}
		In particular, if $N = \lceil 2d\cdot \log d\rceil$,
		\begin{align*}
			\Pb (\Xi = [d]) \geq 1 - d e^{-2 \cdot \log d} 
			= 1 - \frac{1}{d}.
		\end{align*}
		Since we assume $d\geq 2$, it follows that  $\Pb (\Xi = [d]) \geq 1/2$.
	\end{proof}
	
	\begin{proof}[Proof of Lemma \ref{marginalsloglipsmooth}]
		Consider a partition of space such that $\pi = \pi ( \theta,  s )$ with $\theta \in \R^{d_1}$, $s \in \R^{d_2}$, where $d_1 + d_2 = d$. Write $\pi (\theta, s) \propto \exp(-U(\theta, s))$.
		If $U$ is $L$-Lipschitz continuous, then the gradient $\nabla U$ exists almost everywhere and satisfies $\abs{\nabla U(\theta, s)} \leq L$.
		The marginal distribution of $s$ can be expressed as $\pi ( s) \propto \exp(-V(s))$, where
		$$
		V (s) = - \log \left( \int e^{- U(\theta, s)} \mathd \theta\right).
		$$
		Additionally, the conditional distribution given $s$ is
		$$
		\pi (\mathd \theta \mid s) = \frac{e^{-U(\theta,s)}\mathd \theta}{\int e^{-U(\psi,s)} \mathd \psi} .
		$$
		In particular, the gradient of $V$ can be written as 
		$$
		\nabla V (s) = \int \nabla_s U(\theta, s) \pi (\mathd \theta \mid s),
		$$
		where $\nabla_s U(\theta, s)$ represents the gradient of $U(\theta, s)$ with respect to the coordinates of $s$. Consequently, 
		$$
		\abs{\nabla V (s)} \leq \int \abs{\nabla_s U(\theta, s)} \pi (\mathd \theta \mid s) \leq L,
		$$
		which shows that $V$ is also $L$-Lipschitz continuous.
		
		Suppose now that $\theta \in \R$ and $s \in \R^{d-1}$. If $U$ is $L$-smooth, then the Hessian matrix $\nabla^2U$ exists almost everywhere and its spectral norm satisfies
		\begin{equation} \label{spectnorm}
			\norm{\nabla^2 U(\theta, s)}_{\mathrm{spec}} = \sup \left\{ \frac{x^T\nabla^2 U(\theta, s) y}{\abs{x}\abs{y}} : x,y \in \R^d, \, x \neq 0, \, y\neq 0\right\} \leq L.
		\end{equation}
		Let $z \in \R^{d-1}$, $z \neq 0$. Taking $x = (0,z)$ and $y = (1,0)$ in \eqref{spectnorm} gives
		$$
		\abs{\nabla^2_{\theta s} U(\theta, s)} = \sup \left\{ \frac{z^T\nabla^2_{\theta s} U(\theta, s)}{\abs{z}} : z \in \R^{d-1}, \, z \neq 0\right\} \leq L.
		$$
		Moreover, taking $x = y = (0,z)$ and $x = -y = (0,z)$ in \eqref{spectnorm} yields
		$$
		-LI_{d-1} \preceq \nabla_{ss}^2 U(\theta, s) \preceq LI_{d-1}.
		$$
		Writing $z^2 = zz^T$ for the rank-one matrix associated with $z \in \R^{d-1}$, the Hessian matrix  $\nabla^2 V$ admits the decomposition
		\begin{align}
			\nabla^2 V(s) &= \int \nabla^2_{ss} U(\theta, s) \pi (\mathd\theta \mid s) + \left(\int \nabla_s U(\theta, s) \pi (\mathd\theta \mid s)\right)^2 - \int \left(\nabla_s U(\theta, s)\right)^2 \pi (\mathd\theta \mid s) \nonumber \\
			&= \int \nabla^2_{ss} U(\theta, s) \pi (\mathd\theta \mid s) - \int \left( \nabla_s U(\theta, s) - \int \nabla_s U(\theta, s) \pi (\mathd\theta \mid s) \right)^2  \pi (\mathd\theta \mid s) \nonumber  \\
			&= \int \nabla^2_{ss} U(\theta, s) \pi (\mathd\theta \mid s) - \Var_{\theta \sim \pi(\cdot \mid s)} [\nabla_s U(\theta, s)], \label{nabla2v}
		\end{align}
		where $\Var_{\theta \sim \pi(\cdot \mid s)} [\nabla_s U(\theta, s)]$ denotes the covariance matrix of the vector $\nabla_s U(\theta, s)$. Since $\Var_{\theta \sim \pi(\cdot \mid s)} [\nabla_s U(\theta, s)] \succeq 0$, it follows from \eqref{nabla2v} that
		\begin{equation} \label{hessianupper}
			\nabla^2 V(s) \preceq  \int \nabla^2_{ss} U(\theta, s) \pi (\mathd\theta \mid s) \preceq L I_{d-1}.
		\end{equation}
		We now derive an upper bound on the covariance matrix $\Var_{\theta \sim \pi(\cdot \mid s)} [\nabla_s U(\theta, s)]$. If each one-dimensional conditional distribution $\pi(\cdot \mid s)$ satisfies \po{2} with constant $\mathtt{c}_2$, then, for any $z \in \R^{d-1}$,
		\begin{align*}
			z^T \Var_{\theta \sim \pi(\cdot \mid s)} [\nabla_s U(\theta, s)] z &=
			\Var_{\theta \sim \pi(\cdot \mid s)} [z^T\nabla_s U(\theta, s)] \\
			&\leq \frac{1}{\mathtt{c}_2} \cdot \E_{\theta \sim \pi(\cdot \mid s)} \left[\left(z^T\nabla_{\theta s}^2 U(\theta, s)\right)^2\right] \\
			&\leq \frac{\abs{z}^2}{\mathtt{c}_2}  \cdot \E_{\theta \sim \pi(\cdot \mid s)} \left[\abs{\nabla_{\theta s}^2 U(\theta, s)}^2\right] \\
			&\leq \frac{L^2\abs{z}^2}{\mathtt{c}_2} .
		\end{align*}
		It follows that
		$$
		\Var_{\theta \sim \pi(\cdot \mid s)} [\nabla_s U(\theta, s)]  \preceq \frac{L^2}{\mathtt{c}_2} I_{d-1}.
		$$
		Substituting this bound into \eqref{nabla2v}, we obtain
		\begin{equation} \label{hessianlower}
			\nabla^2 V(s) \succeq \int \nabla^2_{ss} U(\theta, s) \pi (\mathd\theta \mid s) - \frac{L^2}{\mathtt{c}_2} I_{d-1} \succeq - \left( L +  \frac{L^2}{\mathtt{c}_2} \right) I_{d-1}.
		\end{equation}
		Combining \eqref{hessianupper} and \eqref{hessianlower}, we conclude that
		$$
		- \left( L +  \frac{L^2}{\mathtt{c}_2} \right) I_{d-1} \preceq \nabla^2 V(s) \preceq LI_{d-1},
		$$
		and therefore $V$ is $L'$-smooth with
		$$
		L' = L +  \frac{L^2}{\mathtt{c}_2}.
		$$
		
	\end{proof}

\end{document}